\documentclass[10pt,a4paper]{amsart}

\usepackage{graphicx,amssymb,color,amscd}
\usepackage[justification=centering]{caption}

\theoremstyle{plain}
\newtheorem{theorem}{Theorem}[section]
\newtheorem{lemma}[theorem]{Lemma}
\newtheorem{proposition}[theorem]{Proposition}
\newtheorem{corollary}[theorem]{Corollary}
\theoremstyle{definition}
\newtheorem{definition}[theorem]{Definition}
\newtheorem{thm}{Theorem}
\theoremstyle{remark}
\newtheorem{notation}[theorem]{Notation} 
\newtheorem{convention}[theorem]{Convention} 
\newtheorem{remark}[theorem]{Remark}
\newtheorem{fact}[thm]{Fact}
\newtheorem{claim}[theorem]{Claim}
\newtheorem*{acknowledgments}{Acknowledgments}
\numberwithin{equation}{section}
\numberwithin{figure}{section}

\newcommand{\bd}{\begin{description}}   
\newcommand{\ed}{\end{description}} 
\newcommand{\ba}{\begin{array}}      \newcommand{\ea}{\end{array}} 
\newcommand{\bc}{\begin{center}}     \newcommand{\ec}{\end{center}} 
\newcommand{\be}{\begin{enumerate}}  \newcommand{\ee}{\end{enumerate}} 
\newcommand{\beq}{\begin{eqnarray}}  \newcommand{\eeq}{\end{eqnarray}} 
\newcommand{\beQ}{\begin{eqnarray*}} \newcommand{\eeQ}{\end{eqnarray*}} 
\newcommand{\bi}{\begin{itemize}}    \newcommand{\ei}{\end{itemize}}

\newcommand{\ov}{\overline}

\newcommand{\s}{\mathcal{S}}

\newcommand{\1}{\mathbf{1}}
\newcommand{\Tube}{\textrm{Tube}}
\newcommand{\Aut}{\textrm{Aut}}
\newcommand{\lr}[1]{\stackrel{\scriptsize{\raisebox{.5ex}[0pt][.5ex]{$#1$}}}{\rightarrow}}
\newcommand\mydot{\ifmmode\, \cdots\else\makebox[1em][c]{\hfil.\hfil.\hfil.}\fi}
\newcommand\mydots{{\cdotp}{\cdotp}{\cdotp}}%

\newcommand{\w}{\textrm{w}} 
\begin{document} 
\title[Arrow Calculus]{Arrow calculus for welded and classical links} 
\author[J.B. Meilhan]{Jean-Baptiste Meilhan} 
\address{Univ. Grenoble Alpes, CNRS, Institut Fourier, F-38000 Grenoble, France}
	 \email{jean-baptiste.meilhan@univ-grenoble-alpes.fr}
\author[A. Yasuhara]{Akira Yasuhara} 
\address{Tsuda University, Kodaira-shi, Tokyo 187-8577, Japan}
	 \email{yasuhara@tsuda.ac.jp}
\begin{abstract} 
We develop a calculus for diagrams of knotted objects. 
We define Arrow presentations, which encode the crossing informations of a diagram into arrows in a way 
somewhat similar to Gauss diagrams,  
and more generally w-tree presentations, which can be seen as `higher order Gauss diagrams'. 
This Arrow calculus is used to develop an analogue of Habiro's 
clasper theory for welded knotted objects, which contain classical  link diagrams as a subset. This provides a 'realization' of Polyak's algebra of arrow diagrams at the welded level, and leads to a characterization of finite type invariants of welded knots and long knots. 
As a corollary, we recover several topological results due to Habiro and 
Shima and to Watanabe on knotted surfaces in $4$-space. 
We also classify welded string links up to homotopy, thus recovering a result of the first author with Audoux, Bellingeri and Wagner. 
\end{abstract} 

\maketitle

\section{Introduction}
A Gauss diagram is a combinatorial object, introduced by M.~Polyak and O.~Viro in \cite{PV} and T.~Fiedler in \cite{fiedler}, which encodes faithfully $1$-dimensional knotted objects in $3$-space. 
To a knot diagram, one associates a Gauss diagram by connecting, on a copy of $S^1$, the two preimages of each crossing by an arrow, oriented from the over- to the under-passing strand and labeled by the sign of the crossing. 
Gauss diagrams form a powerful tool for studying knots and their invariants. In particular, a result of M.~Goussarov \cite{GPV} states that any finite type (Goussarov-Vassiliev) knot invariant admits a Gauss diagram formula, i.e. can be expressed as a weighted count of arrow configurations in a Gauss diagram. 
A remarkable feature of this result is that, although it concerns classical knots, its proof heavily relies on \emph{virtual knot theory}.
Indeed, Gauss diagrams are inherently related to virtual knots, since an arbitrary Gauss diagram doesn't always represent a classical knot, but a virtual one \cite{GPV,Kauffman}. 

More recently, further topological applications of virtual knot theory arose from its \emph{welded} quotient, where one allows a strand to pass over a virtual crossing \cite{FRR}. 
This quotient is completely natural from the virtual knot group viewpoint, which naturally satisfies this additional local move. Hence all virtual invariants derived from the knot group, such as the Alexander polynomial or Milnor invariants, are intrinsically invariants of welded knotted objects. 
Welded theory is also natural by the fact that classical knots and (string) links can be 'embedded' in their welded counterparts. 
The topological significance of welded theory was enlightened by  S.~Satoh \cite{Satoh}; 
building on early works of  T.~Yajima \cite{yajima}, he defined the so-called Tube map, which `inflates' welded diagrams into ribbon knotted surfaces in dimension $4$. 
Using the Tube map, welded theory was successfully used in \cite{ABMW} to classify ribbon knotted annuli and tori up to link-homotopy (for knotted annuli, it was later shown that the ribbon case can be used to give a general link-homotopy classification \cite{AMW}).

\medskip

In this paper, we develop an \emph{arrow calculus} for welded knotted objects, which can be regarded as a kind of `higher order Gauss diagram' theory. 
We first recast the notion of Gauss diagram into so-called \emph{Arrow presentations} for classical and welded knotted objects. 
Unlike Gauss diagrams, which are `abstract' objects, Arrow presentations are planar immersed arrows which `interact' with knotted  diagrams. 
They satisfy a set of \emph{Arrow moves}, which we prove to be complete, in the following sense. 
\begin{thm}[Thm.~\ref{thm:main1}]\label{thm1}
Two Arrow presentations represent equivalent diagrams if and only if they are related by Arrow moves.
\end{thm}
We stress that, unlike Gauss diagrams analogues of Reidemeister moves, which involve rather delicate compatibility conditions in terms of the arrow signs and local strands orientations, Arrow moves involve no  such restrictions. \\
The main advantage of this calculus, however, is that it generalizes to `higher orders'. This relies on the notion of 
 \emph{w-tree presentation}, where arrows 
 are generalized to oriented trees, which can thus be thought of as `higher order Gauss diagrams'. Arrow moves are then extended to a calculus of \emph{w-tree moves}, i.e. we have a w-tree version of Theorem \ref{thm1}. 

\medskip

Arrow calculus should also be regarded as a welded version of the Goussarov-Habiro theory \cite{Habiro,Gusarov:94}, solving  partially a problem set by M.~Polyak in \cite[Problem 2.25]{ohtsukipb}. 
In \cite{Habiro}, Habiro introduced the notion of clasper for (classical) knotted objects, which is a kind of embedded graph carrying a surgery instruction. 
A striking result is that clasper theory gives a topological characterization of the information carried by finite type invariants of knots.  
More precisely, Habiro used claspers to define the $C_k$-equivalence relation, for any integer $k\ge 1$, and showed that two knots share all finite type invariants up to degree $<k$ if and only if they are $C_k$-equivalent.  This result was also independently obtained by Goussarov in \cite{Gusarov:94}. 
In the present paper, we use w-tree presentations to define a notion of $\w_k$-equivalence. 
We observe that two $\w_k$-equivalent welded knotted objects share all finite type invariants of degree $<k$,  and prove that the converse holds for welded knots and long knots. 
More precisely, we use Arrow calculus to show the following. 
\begin{thm}[Thm.~\ref{thm:wk} and Cor.~\ref{cor:ftiwklong}]\label{thm2}
Any welded knot is $\w_k$-equivalent to the unknot, for \emph{any} $k\ge 1$. 
Hence there is no non-trivial finite type invariant of welded knots. 
\end{thm}
\begin{thm}[Cor.~\ref{cor:ftiwk}]\label{thm3}
The following assertions are equivalent, for any $k\ge 1$: 
\begin{enumerate}
\item[$\circ$]  two welded long knots are $\w_k$-equivalent, 
\item[$\circ$]  two welded long knots share all finite type invariants of degree $<k$, 
\item[$\circ$]  two welded long knots have same invariants 
$\{ \alpha_i \}$ for $2\le i\le k$.
\end{enumerate}
\end{thm}
\noindent Here, the invariants $\alpha_i$ are given by the coefficients of the power series expansion at $t=1$ of the normalized Alexander polynomial. 

From the finite type point of view, w-trees can thus be regarded as a `realization' of the space of oriented diagrams 
introduced in \cite{WKO1}, 
where the universal invariant of welded (long) knots takes its values, 
and which is a quotient of the Polyak algebra \cite{polyak_arrow}. 
This  is similar to clasper theory, which provide a topological realization of Jacobi diagrams. See Sections \ref{sec:wslfti} and \ref{sec:virtual} for further comments. 
\\ 
We note that Theorem \ref{thm2} and the equivalence (2)$\Leftrightarrow$(3) of Theorem \ref{thm3} were independently shown for \emph{rational-valued} finite type invariants by D.~Bar-Natan and S.~Dancso \cite{WKO1}. 
Our results hold in the general case, i.e. for invariants valued in any abelian group. 
We also show that welded long knots up to $\w_k$-equivalence form a finitely generated free abelian group, see Corollary \ref{cor:abelian}. 

\medskip

Using Satoh's Tube map, we can promote these results to topological ones. More precisely, we obtain that there is no non-trivial finite type invariant of ribbon torus-knots (Cor.~\ref{cor:topo1}), and reprove a result of K.~Habiro and A.~Shima \cite{HS} stating that  
finite type invariants of ribbon $2$-knots are determined by the (normalized) Alexander polynomial (Cor.~\ref{cor:topo2}). 
Moreover, we show that Theorem \ref{thm3} implies a result of T.~Watanabe \cite{watanabe} which characterizes topologically finite type invariants of ribbon $2$-knots.
See Section \ref{sec:watanabe}. 

\medskip

We also develop a version of Arrow calculus \emph{up to homotopy}. 
Here, the notion of homotopy for welded diagrams is generated by the \emph{self-(de)virtualization move}, which replaces a classical crossing between two strands of a same component by a virtual one, or vice-versa. 
We use the homotopy Arrow calculus to prove the following. 
\begin{thm}[Cor.~\ref{thm:wsl}]\label{thm4}
Welded string links are classified up to homotopy by welded Milnor invariants. 
\end{thm}
This result, which is a generalization of Habegger-Lin's classification of string links up to link-homotopy \cite{HL}, was first shown by B.~Audoux, P.~Bellingeri, E.~Wagner and the first author in \cite{ABMW}. Our version is stronger in that it gives, in terms of w-trees and welded Milnor invariants, an explicit representative for the homotopy class of a welded string link, see Theorem \ref{thm:hrep}.  
Moreover, this result can be used to give homotopy classifications of ribbon annuli and torus-links, as shown in \cite{ABMW}.

\medskip 

The rest of this paper is organized as follows. 

We recall in Section \ref{sec:basics} the basics on classical and welded knotted objects, and the connection to ribbon knotted objects in dimension $4$. 
In Section \ref{sec:w}, we give the main definition of this paper, introducing w-arrows and w-trees. 
We then focus on w-arrows in Section \ref{sec:warrsurg}. We  define Arrow presentations and Arrow moves, and prove Theorem \ref{thm1}. 
The relation to Gauss diagrams is also discussed in more details in Section \ref{sec:GD1}. 
Next, in Section \ref{sec:wtreesurg} we turn to w-trees. We define the Expansion move (E), which leads to the notion of w-tree presentation, and we provide a collection of moves on such presentations. 
In Section \ref{sec:invariants}, we give the definitions and some properties of the welded extensions of the knot group, the normalized Alexander polynomial, and Milnor invariants. We also review the finite type invariant theory for welded knotted objects. 
The $\w_k$-equivalence relation is introduced and studied in Section \ref{sec:wkequiv}.  We also clarify there the relation to finite type invariants and to Habiro's $C_n$-equivalence. 
Theorems \ref{thm2} and \ref{thm3} are proved in Section \ref{sec:wk_knots}.  
In Section \ref{sec:homotopy}, we consider Arrow calculus up to homotopy, and prove Theorem \ref{thm4}. 
We close this paper with Section \ref{sec:thisistheend}, where we gather several comments, questions and remarks. 
In particular, we prove in Section \ref{sec:watanabe} the topological consequences of our results, stated above. 

\begin{acknowledgments}
The authors would like to thank Benjamin Audoux for stimulating conversations, 
and Haruko A.~Miyazawa for her useful comments. 
This paper was completed during a visit of first author at Tsuda University, Tokyo, whose hospitality and support is warmly acknowledged. 
The second author is partially supported by a Grant-in-Aid for Scientific Research (C) 
($\#$17K05264) of the Japan Society for the Promotion of Science.
\end{acknowledgments}

\section{A quick review of classical and welded knotted objects}\label{sec:basics}

\subsection{Basic definitions}

A \emph{classical knotted object} is the image of an embedding of some oriented $1$-manifold in $3$-dimensional space. 
Typical examples include knots and links, braids, string links, and more generally tangles. 
It is well known that such embeddings are faithfully represented by a generic planar projection, where the only singularities are 
transverse double points endowed with a diagrammatic over/under information, as on the left-hand side of Figure \ref{fig:cross}, 
modulo Reidemeister moves I, II and III.  

This diagrammatic realization of classical knotted objects generalizes to virtual and welded knotted objects, as we briefly outline below.

A \emph{virtual diagram} is an immersion of some oriented $1$-manifold in the plane, whose singularities are a finite number of transverse double points that are labeled, either as a \emph{classical crossing} 
or as a \emph{virtual crossing}, as shown in Figure \ref{fig:cross}. 
\begin{figure}[!h]
  \includegraphics[scale=1]{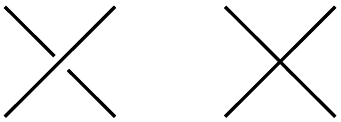}
  \caption{A classical and a virtual crossing.} \label{fig:cross}
\end{figure}
\begin{convention}
Note that we do not use here the usual drawing convention for virtual crossings, with a circle around the corresponding double point. 
\end{convention} 

There are three classes of local moves that one considers on virtual diagrams: 
\begin{itemize}
 \item[$\circ$]  the three classical Reidemeister moves, 
 \item[$\circ$]  the three virtual Reidemeister moves, which are the exact analogues of the classical ones with all classical crossings replaced by virtual ones, 
 \item[$\circ$]  the Mixed Reidemeister move, shown on the left-hand side of Figure \ref{fig:moves}. 
\end{itemize}
We call these three classes of moves the \emph{generalized Reidemeister moves}. 
\begin{figure}[!h]
  \includegraphics[scale=1]{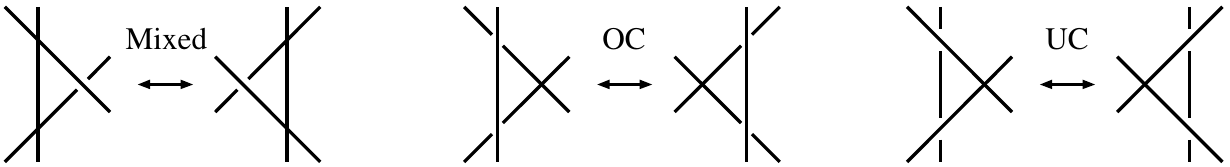}
  \caption{The Mixed, OC and UC moves on virtual diagrams}\label{fig:moves}
\end{figure}
A \emph{virtual knotted object} is the equivalence class of a virtual diagram under planar isotopy and generalized Reidemeister moves.  
This notion was introduced by Kauffman in \cite{Kauffman}, where we refer the reader for a much more detailed treatment. 

Recall that generalized Reidemeister moves in particular imply the so-called \emph{detour move}, 
which replaces an arc passing through a number of virtual crossings 
by any other such arc, with same endpoints.  

Recall also that there are two `forbidden' local moves, called \emph{OC and UC moves} (for Overcrossings and Undercrossings Commute), 
as illustrated in Figure \ref{fig:moves}. 

In this paper, we shall rather consider the following natural quotient of virtual theory. 
\begin{definition}\label{def:welded}
A \emph{welded knotted object} is the equivalence class of a virtual diagram under planar isotopy, generalized Reidemeister moves and OC moves.   
\end{definition}
There are several reasons that make this notion both natural and interesting.  The virtual knot group introduced by Kauffman in \cite{Kauffman} at the early stages of virtual knot theory, is intrasically a welded invariants. 
As a consequence, the virtual extensions of classical invariants derived from (quotients of) the fundamental group 
are in fact welded invariants, see Section \ref{sec:invariants}. 
Another, topological motivation is the relation with ribbon knotted objects in codimension $2$, see Section \ref{sec:ribbon}.

\medskip

In what follows, we will be mainly interested in \emph{welded links and welded string links}, which are the welded extensions of classical link and string link diagrams. 
Recall that, roughly speaking, an $n$-component welded string link is a diagram made of $n$ arcs properly immersed in a square with $n$ points marked on the lower and upper faces, such that the $k$th arc runs from the $k$th lower to the $k$th upper marked point. A $1$-component string link is often called \emph{long knot} in the literature -- we shall use this terminology here as well. 

Welded (string) links are a genuine extension of classical (string) links, in the sense that the latter can de 'embedded' into the former ones. 
This is shown strictly as in the knot case \cite[Thm.1.B]{GPV}, and actually also holds for virtual objects. 

\begin{convention}
 In the rest of this paper, by `diagram' we will implicitly mean an oriented diagram, containing classical and/or virtual crossings, 
 and the natural equivalence relation on diagrams will be that of Definition \ref{def:welded}. 
 We shall sometimes use the terminology `welded diagram' to emphasize this fact. 
 As noted above, this includes in particular classical (string) link diagrams. 
\end{convention}

\begin{remark}\label{rem:wdetour}
 Notice that the OC move, together with generalized Reidemeister moves, implies a welded version of the detour move, called \emph{w-detour move}, 
 which replaces an arc passing through a number of over-crossings 
 by any other such arc, with same endpoints. This is proved strictly as for the detour move, the OC move playing the role of the Mixed move.  
\end{remark}

\subsection{Welded theory and ribbon knotted objects in codimension $2$}\label{sec:ribbon}

As already indicated, one of the main interests of welded knot theory is that it allows to study certain knotted surfaces in $4$-space. 
As a matter of fact, the main results of this paper will have such topological applications, so we briefly review these objects and their connection to welded theory. 

Recall that a \emph{ribbon immersion} of a $3$-manifold $M$ in $4$-space is an immersion admitting only ribbon singularities, 
which are $2$-disks with two preimages, one being embedded in the interior of $M$, and the other being properly embedded. 

A \emph{ribbon $2$-knot} is the boundary of a ribbon immersed $3$-ball in $4$-space, and a  \emph{ribbon torus-knot} is, 
likewise, the boundary of a ribbon immersed solid torus in $4$-space. 
More generally, by \emph{ribbon knotted object}, we mean a knotted surface obtained as the boundary of some ribbon immersed $3$-manifold in $4$-space.

Using works of T.~Yajima \cite{yajima}, S.~ Satoh defined in \cite{Satoh} a surjective \emph{Tube map}, from welded diagrams  to ribbon $2$-knotted  objects. Roughly speaking, the Tube map assigns, to each classical crossing of a diagram, a pair of locally linked annuli in a $4$-ball, as shown in \cite[Fig.~6]{Satoh}; 
next, it only remains to connect these annuli to one another by unlinked annuli, as prescribed by the diagram.
Although not injective in general,\footnote{The Tube map is not injective for welded knots \cite{IK}, but is injective for welded braids \cite{BH} and welded string links up to homotopy \cite{ABMW}. } the Tube map acts faithfully on the `fundamental group'. 
This key fact, which will be made precise in Remark \ref{rem:faithful}, will allow to draw several topological consequences from our diagrammatic results.  See Section \ref{sec:watanabe}. 

\begin{remark}\label{rem:iwanttotakeyouhigher}
One can more generally define $k$-dimensional ribbon knotted objects in codimension $2$, for any $k\ge 2$, and the Tube map generalizes straightforwardly to a surjective map from welded diagrams to $k$-dimensional ribbon knotted objects. See for example \cite{AMW}. 
As a matter of fact, most of the topological results of this paper extend freely to ribbon knotted objects in codimension $2$. 
\end{remark}

\section{w-arrows and w-trees} \label{sec:w}

Let $D$ be a diagram.   
The following is the main definition of this paper.
\begin{definition}
A \emph{$\w$-tree} for $D$ is a connected uni-trivalent tree $T$, immersed in the plane of the diagram such that: 
\begin{itemize}
 \item[$\circ$]  the trivalent vertices of $T$ are pairwise disjoint and disjoint from $D$, 
 \item[$\circ$]  the univalent vertices of $T$ are pairwise disjoint and are contained in $D\setminus \{\textrm{crossings of $D$}\}$, 
 \item[$\circ$]  all edges of $T$ are oriented, such that each trivalent vertex has two ingoing and one outgoing edge, 
 \item[$\circ$]  we allow virtual crossings between edges of $T$, and between $D$ and edges of $T$, but classical crossings involving $T$ are not allowed,    
 \item[$\circ$]  each edge of $T$ is assigned a number (possibly zero) of decorations $\bullet$, called \emph{twists}, which are disjoint from all vertices and crossings, and subject to the involutive rule 
     \begin{center}
    \includegraphics[scale=1.1]{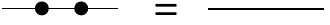} 
    \end{center} 
 \end{itemize}
 \noindent A w-tree with a single edge is called a \emph{w-arrow}. 
\end{definition}
For a union of w-trees for $D$, vertices are assumed to be pairwise disjoint, and all crossings among edges are assumed to be virtual.  
See Figure \ref{fig:extree} for an example. 
\begin{figure}[!h]
  \includegraphics[scale=1]{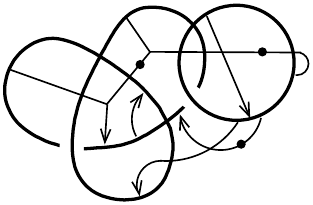} 
  \caption{Example of a union of w-trees}\label{fig:extree}
\end{figure}

We call \emph{tails} the univalent vertices of $T$ with outgoing edges, 
and we call the \emph{head} the unique univalent vertex with an ingoing edge. 
We will call \emph{endpoint} any univalent vertex of $T$, when we do not need to distinguish between tails and head. 
The edge which is incident to the head is called \emph{terminal}.

Two endpoints of a union of w-trees for $D$ are called \emph{adjacent} if, when travelling along $D$,  
these two endpoints are met consecutively, without encountering any crossing or endpoint.  

\begin{remark}
Note that, given a uni-trivalent tree, picking a univalent vertex as the head uniquely determines an orientation on all edges respecting the above rule. 
Thus, we usually only indicate the orientation on w-trees at the terminal edge. 
However, it will occasionnally be useful to indicate the orientation on other edges, for example when drawing local pictures. 
\end{remark}

\begin{definition}
Let $k\ge 1$ be an integer. 
A \emph{$\w$-tree of degree $k$}, or \emph{$\w_k$-tree}, for $D$ is a w-tree for $D$ with $k$ tails. 
\end{definition}

\begin{convention}
We will use the following drawing conventions. 
Diagrams are drawn with bold lines, while w-trees are drawn with thin lines. 
See Figure \ref{fig:extree}. 
We shall also use the symbol $\circ$ to describe a w-tree that \emph{may or may not} contain a twist at the indicated edge:  
 \begin{center}
  \includegraphics[scale=1]{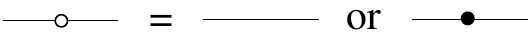} 
 \end{center}
\end{convention}

\section{Arrow presentations of diagrams} \label{sec:warrsurg}

In this section, we focus on w-arrows. 
We explain how w-arrows carry `surgery' instructions on diagrams, so that they provide a way to encode diagrams. 
A complete set of moves is provided, relating any two w-arrow presentations of equivalent diagrams. 
The relation to the theory of Gauss diagrams is also discussed. 

\subsection{Surgery along w-arrows}
Let $A$ be a union of w-arrows for a diagram $D$. 
\emph{Surgery along $A$} yields a new diagram, denoted by $D_A$, which is defined as follows. 

Suppose that there is a disk in the plane that intersects $D\cup A$ as shown in Figure \ref{fig:surgery}. 
The figure then represents the result of surgery along $A$ on $D$. 
\begin{figure}[!h]
  \includegraphics[scale=1]{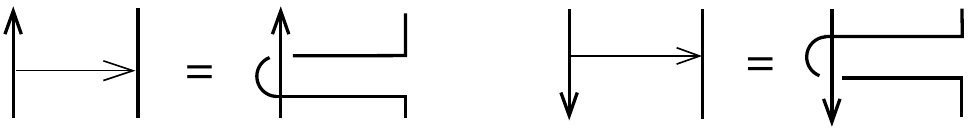} 
  \caption{Surgery along a w-arrow}\label{fig:surgery}
\end{figure}
\noindent 
We emphasize the fact that the orientation of the portion of diagram containing the tail needs to be specified to define the surgery move. 

If some w-arrow of $A$ intersects the diagram $D$ (at some virtual crossing disjoint from its endpoints), then this introduces pairs of virtual crossings as indicated on the left-hand side of the figure below. 
Likewise, the right-hand side of the figure indicates the rule when two portions of  (possibly of the same) w-arrow(s) of $A$ intersect.  
\begin{center}
  \includegraphics{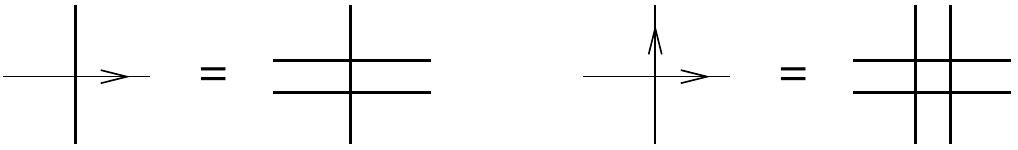}
\end{center}

Finally, if some w-arrow of $A$ contains some twists, we simply insert virtual crossings accordingly, as indicated below:
\begin{center}
  \includegraphics{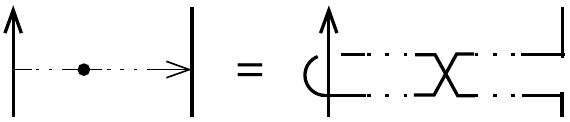}
\end{center}

\noindent 
Note that this is compatible with the involutive rule for twists by the virtual Reidemeister II move, as shown below. 
    \begin{center}
    \includegraphics{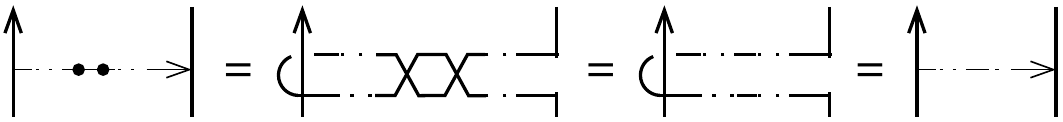}
    \end{center} 
    
An example is given in Figure \ref{fig:exemple}.  
\begin{figure}[!h]
  \includegraphics[scale=1]{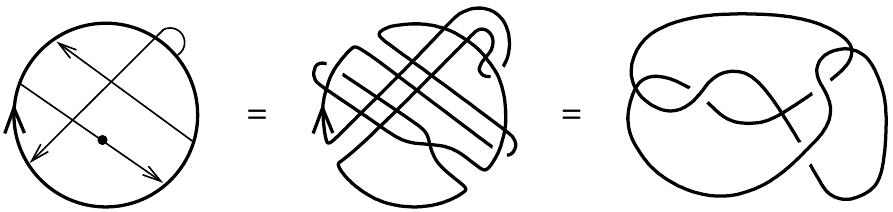}
  \caption{An example of diagram obtained by surgery along w-arrows}\label{fig:exemple}
\end{figure}

\subsection{Arrow presentations}

Having defined surgery along w-arrows, we are led to the following. 
\begin{definition}\label{def:arrowpres}
An \emph{Arrow presentation} for a diagram $D$ is a pair $(V,A)$ of a diagram $V$ \emph{without classical crossings} and a collection of w-arrows $A$ for $V$, such that surgery on $V$ along $A$ yields the diagram $D$. \\
We say that two Arrow presentations are \emph{equivalent} if the surgeries yield equivalent diagrams.  
We will simply denote this equivalence by $=$. 
\end{definition}
In the next section, we address the problem of generating this equivalence relation by local moves on Arrow presentations. 

As Figure \ref{fig:wcross} illustrates, surgery along a w-arrow is equivalent to a \emph{devirtualization move}, which is a local move that replaces a virtual crossing by a classical one. 
\begin{figure}[!h]
  \includegraphics[scale=1]{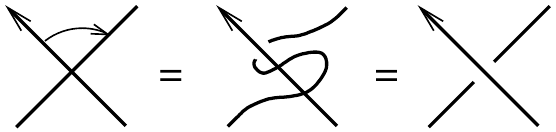} 
  \caption{Surgery along a w-arrow is a devirtualization move. }\label{fig:wcross}
\end{figure}
This observation implies the following.
\begin{proposition}\label{prop:wApres}
 Any diagram admits an Arrow presentation. 
\end{proposition}
More precisely, for a diagram $D$, there is a uniquely defined Arrow presentation $(V_D,A)$  
which is obtained by applying the rule of Figure \ref{fig:wcross} at each (classical) crossing. 
Note that $V_D$ is obtained from $D$ by replacing all classical crossings by virtual ones. 
\begin{definition}
We call the pair $(V_D,A)$ the \emph{canonical Arrow presentation} of the diagram $D$. 
\end{definition}
\noindent For example, for the diagram of the trefoil show in Figure \ref{fig:trefoil}, the canonical Arrow presentation is given in the center of the figure. 

\subsection{Arrow moves}\label{sec:arrow_moves}

 \emph{Arrow moves} are the following six types of local moves among Arrow presentations. 
 \begin{enumerate}
  \item[(1) ] Virtual Isotopy. 
    Virtual Reidemeister moves involving edges of w-arrows and/or strands of diagram, together with the following local moves:\footnote{Here, in the figures, the vertical strand is either a portion of diagram or of a w-arrow.}
    \begin{center}
    \includegraphics{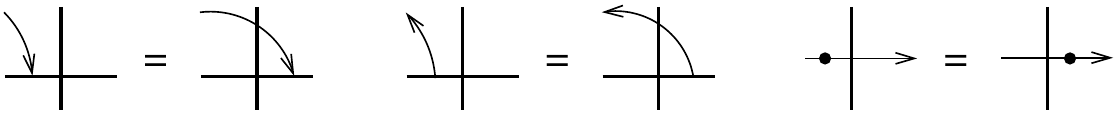} 
    \end{center}
  \item[(2) ] Head/Tail Reversal.
    \begin{center}
    \includegraphics[scale=0.9]{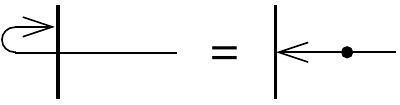} \qquad \qquad
    \includegraphics[scale=0.9]{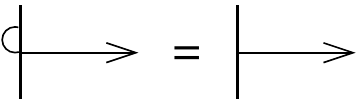} 
    \end{center}
  \item[(3) ] Tails Exchange.
    \begin{center}
    \includegraphics[scale=1]{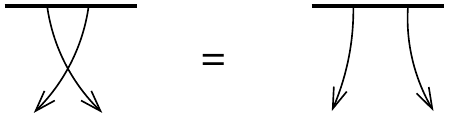} 
    \end{center}
  \item[(4) ] Isolated Arrow.
    \begin{center}
    \includegraphics[scale=0.8]{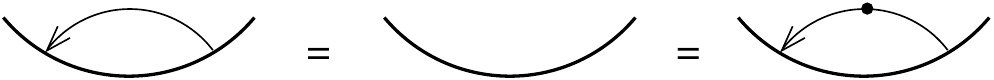} 
    \end{center}
  \item[(5) ] Inverse.
    \begin{center}
    \includegraphics[scale=0.9]{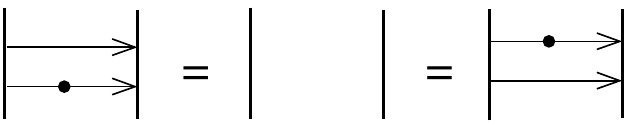} 
    \end{center}
  \item[(6) ] Slide.
    \begin{center}
    \includegraphics[scale=1]{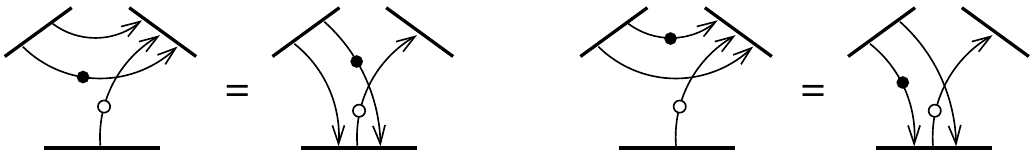} 
    \end{center}
 \end{enumerate}

\begin{lemma}\label{lem:wamoves}
Arrow moves yield equivalent Arrow presentations. 
\end{lemma}
\begin{proof}
 Virtual Isotopy moves (1) are easy consequences of the surgery definition of w-arrows and virtual Reidemeister moves. 
 This is clear for the Reidemeister-type moves, since all such moves locally involve only virtual crossings. 
 The remaining local moves essentially follow  from detour moves. 
 For example, the figure below illustrates the proof of one instance of the second move, for one choice of orientation at the tail: 
    \begin{center}
    \includegraphics{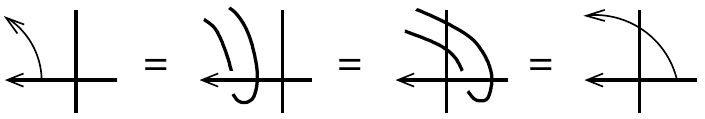}
    \end{center}  
\noindent All other moves of (1) are given likewise by virtual Reidemeister moves. 

Having proved this first sets of moves, we can freely use them to simplify the proof of the remaining moves. 
For example, we can freely assume that the w-arrow involved in the Reversal move (2) is either as 
shown on the left-hand side of Figure \ref{fig:tr} below, or differs from this figure by a single twist. 
The proof of the Tail Reversal move is given in Figure \ref{fig:tr} in the case where the w-arrow has no twist and the strand is oriented upwards (in the figure of the lemma). 
\begin{figure}[!h]
   \includegraphics[scale=0.8]{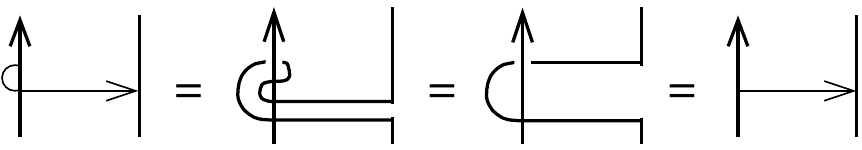}
  \caption{Proving the Tail Reversal move}\label{fig:tr}
\end{figure}
It only uses the definition of a w-arrow and the virtual Reidemeister II move. 
The other cases are similar, and left to the reader. 
\\
Likewise, we only prove Head Reversal in Figure \ref{fig:hr} when the w-arrow has no twist. Note that the Tail Reversal and Isotopy moves allow us to chose the strand orientation as depicted. 
\begin{figure}[!h]
   \includegraphics[scale=0.8]{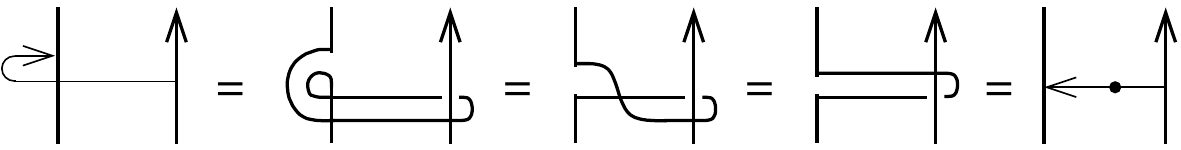}
  \caption{Proving the Head Reversal move}\label{fig:hr}
\end{figure}
\noindent The identities in the figure follow from elementary applications of generalized Reidemeister moves. 

Figure \ref{fig:tc} shows (3). There, the second and fourth identities are applications of the detour move, while the third move uses the OC move. 
 \begin{figure}[!h]
   \includegraphics[scale=1]{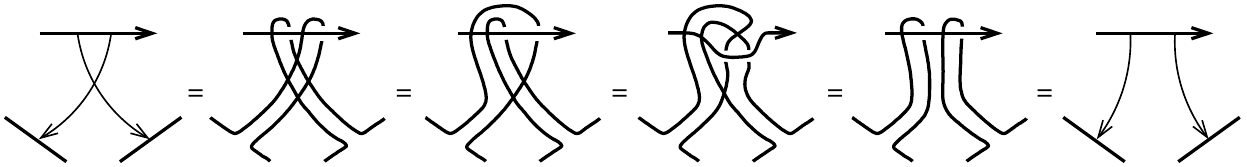}
   \caption{Proving the Tails Exchange move}\label{fig:tc}
\end{figure}
\noindent In Figure \ref{fig:tc}, we had to choose a local orientation for the upper strand.
This implies the result for the other choice of orientation, by using the Tail Reversal move (2).

Moves  (4) and (5) are direct consequences of the definition, and are left to the reader.

Finaly, we prove (6). 
We only show here the first version of the move, the second one being strictly similar. 
There are \emph{a priori} several choices of local orientations to consider, 
which are all declined in two versions, depending on whether we insert a twist on the $\circ$-marked w-arrow or not.  
Figure \ref{fig:slide} illustrates the proof for one choice of orientation, in the case where no twist is inserted. 
The sequence of identities in this figure is given as follows: the second and third identities use isotopies and detour moves, 
the fourth (vertical) one uses the OC move, then followed by  isotopies and detour moves which give the fifth equality. 
The final step uses the Tails Exchange move (3). 
 \begin{figure}[!h]
   \includegraphics[scale=1]{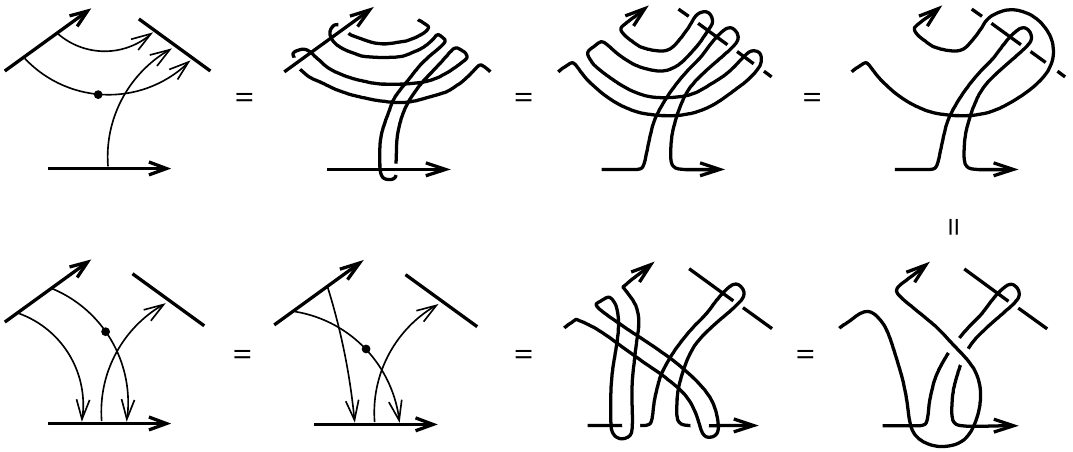}
  \caption{Proving the Slide move}\label{fig:slide}
\end{figure}

\noindent Now, notice that the exact same proof applies in the case where there is a twist on the $\circ$-marked w-arrow. 
Moreover, if we change the local orientation of, say, the bottom strand in the figure, the result follows from the previous case 
by the Reversal move (2), the Tails Exchange move (3) and twist involutivity, as the following picture indicates: 
\begin{center}
  \includegraphics[scale=1]{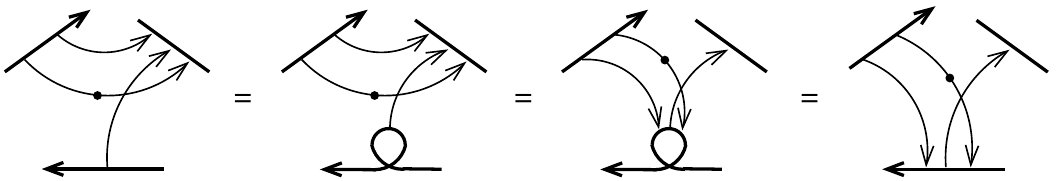}
\end{center}
We leave it to the reader to check that, similarly, all other choices of local orientations follow from the first one. 
\end{proof}

The main result of this section is that this set of moves is complete. 
\begin{theorem}\label{thm:main1}
Two Arrow presentations represent equivalent diagrams if and only if they are related by Arrow moves.
\end{theorem}

The \emph{if} part of the statement is shown in Lemma \ref{lem:wamoves}. 
In order to prove the \emph{only if} part, we will need the following. 

\begin{lemma}\label{lem:equivwA}
If two diagrams are equivalent, then their canonical Arrow presentations are related by Arrow moves.  
\end{lemma}
\begin{proof}
It suffices to show that generalized Reidemeister moves and OC moves are realized by Arrow moves among canonical Arrow presentations. 

Virtual Reidemeister moves and the Mixed move follow from Virtual Isotopy moves (1). For example, the case of the Mixed move is illustrated in Figure \ref{fig:mixed_proof} (the argument holds for any choice of orientation).
 \begin{figure}[!h]
  \includegraphics[scale=0.90]{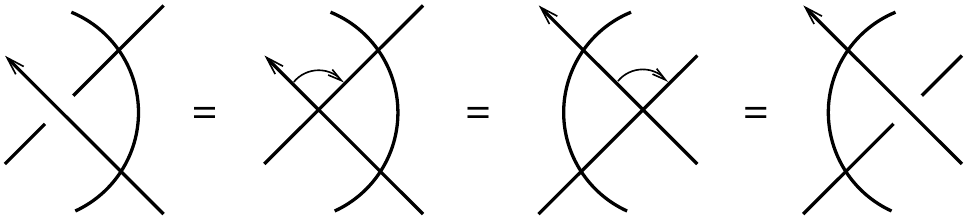}
  \caption{Realizing the Mixed move by Arrow moves}\label{fig:mixed_proof}
\end{figure}

The OC move is, expectedly, essentially a consequence of the Tails Exchange move (3). 
More precisely, Figure \ref{fig:OC_proof} shows how applying the Tails Exchange together with Isotopy moves (1), followed by Tail Reversal moves (2), and further Isotopy moves, realizes the OC move. 
 \begin{figure}[!h]
  \includegraphics[scale=0.90]{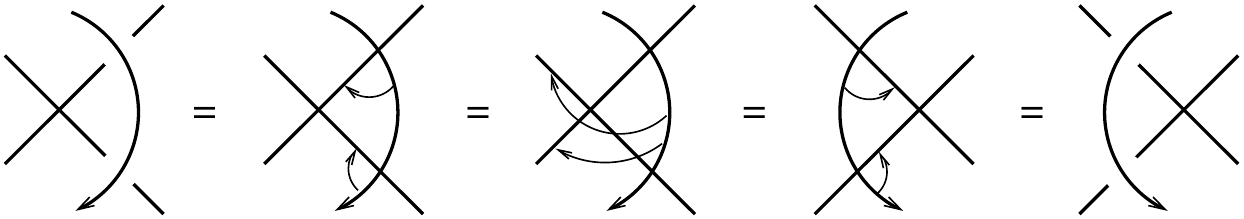}
  \caption{Realizing the OC move by Arrow moves}\label{fig:OC_proof}
\end{figure}

We now turn to classical Reidemeister moves. The proof for the Reidemeister I move is illustrated in Figure \ref{fig:R1_proof}. There, the second equality uses move (1), while the third equality uses the Isolated Arrow move (4). (More precisely, one has to consider both orientations in the figure, as well as the opposite crossing, but these other cases are similar.)
 \begin{figure}[!h]
  \includegraphics[scale=0.90]{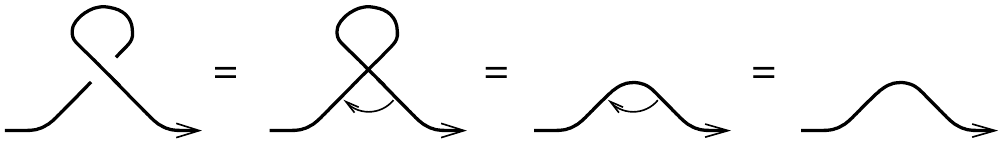}
  \caption{Realizing the Reidemeister I move by Arrow moves}\label{fig:R1_proof}
\end{figure}
The proof for the Reidemeister II move is shown in Figure \ref{fig:R2_proof}, where the second equality uses moves 
(1) and the Head Reversal move (2), and the third equality uses the Inverse move (5).
 \begin{figure}[!h]
  \includegraphics[scale=1]{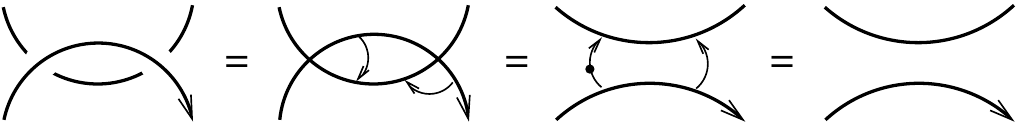}
  \caption{Realizing the Reidemeister II move by Arrow moves}\label{fig:R2_proof}
\end{figure}
Finally, for the Reidemeister move III, we first note that, although there are \emph{a priori} eight choices of orientation to be considered, Polyak showed that only one is necessary \cite{PolyakR}. 
We consider this move in Figure \ref{fig:R3_proof}. 
\begin{figure}[!h]
  \includegraphics[scale=1]{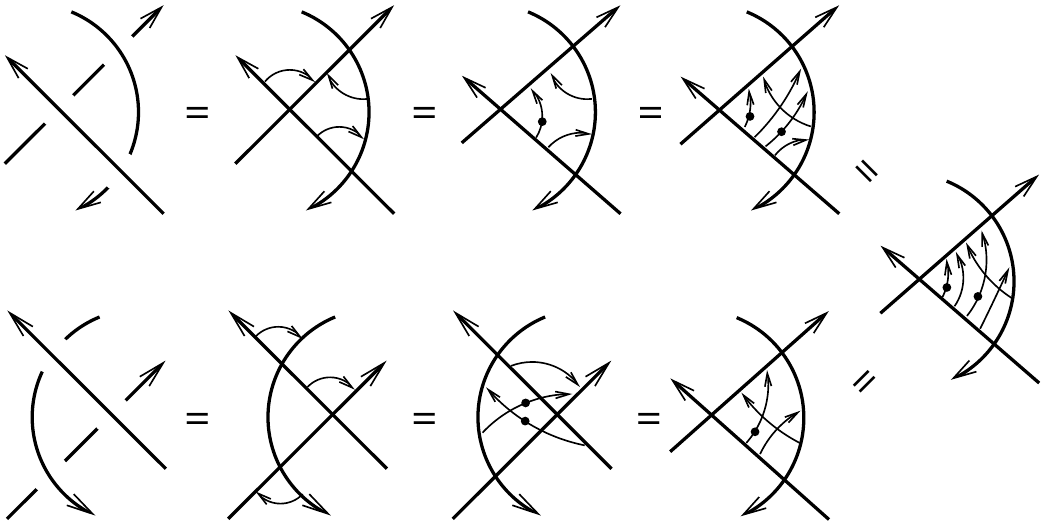}
  \caption{Realizing the Reidemeister III move by Arrow moves}\label{fig:R3_proof}
\end{figure}
There, the second equality uses the Reversal and Isotopy moves (2) and (1), the third equality uses the Inverse move (5), and the fourth one uses the Slide move (6) as well as the Tails Exchange move (3). Then the fifth equality uses the Inverse move back again, the sixth equality uses the Reversal, Isotopy and Tails Exchange moves, and the seventh one uses further Reversal and Isotopy moves. 
\end{proof}
\begin{remark}
We note from the above proof that some of the Arrow moves appear as essential analogues of the generalized Reidemeister moves: the Isolated move (4) gives Reidemeister I move, while the Inverse move (5) and Slide move (6) give Reidemeister II and III, respectively. Finaly, the Tails Exchange move (3) corresponds to the OC move.  
\end{remark}

We can now prove the main result of this section. 
\begin{proof}[Proof of Theorem \ref{thm:main1}]
As already mentioned, it suffices to prove the \emph{only if} part. 
Observe that, given a diagram $D$, any Arrow presentation of $D$ is equivalent to the canonical Arrow presentation of some diagram. 
Indeed, by the involutivity of twists and the Head Reversal move (2), we can assume that the Arrow presentation of $D$ contains no twist. 
We can then apply Isotopy and Tail Reversal moves (1) and (2) to assume that each w-arrow is contained in a disk where it looks as on the left-hand side of Figure \ref{fig:surgery}; by using virtual Reidemeister moves II, we can actually assume that it is next to a (virtual) crossing, as on the left-hand side of Figure \ref{fig:wcross}. The resulting Arrow presentation is thus a canonical Arrow presentation of some diagram (which is equivalent to $D$, by Lemma \ref{lem:wamoves}). 

Now, consider two equivalent diagrams, and pick any Arrow presentations for these diagrams. 
By the previous observation, these Arrow presentations are equivalent to canonical Arrow presentations of equivalent diagrams. The result then follows from Lemma \ref{lem:equivwA}. 
\end{proof}

\subsection{Relation to Gauss diagrams}\label{sec:GD1}

Although similar-looking and closely related, w-arrows are not to be confused with arrows of Gauss diagrams. 
In particular, the signs on arrows of a Gauss diagram are not equivalent to twists on w-arrows. 
Indeed, the sign of the crossing defined by a w-arrow relies on the local orientation of the strand where its head is attached. 
The local orientation at the tail, however, is irrelevant. 
Let us clarify here the relationship between these two objects. 

Given an Arrow presentation $(V,A)$ for some diagram $K$ (of, say, a knot) one can always turn it by Arrow moves into 
an Arrow presentation $(V_0,A_0)$, where $V_0$ is a trivial diagram, with no crossing. 
See for example the case of the trefoil in Figure \ref{fig:trefoil}. 
 \begin{figure}[!h]
  \includegraphics[scale=1]{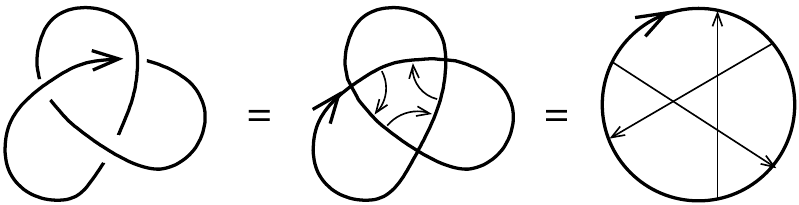} 
  \caption{The right-handed trefoil as obtained by surgery on w-arrows}\label{fig:trefoil}
\end{figure}
There is a unique Gauss diagram for $K$ associated to $(V_0,A_0)$, which is simply obtained by the following rule.  
First, each w-arrow in $A_0$ enherits a sign, which is $+$ (resp. $-$) if, when running along $V_0$ following the orientation, 
the head is attached to the right-hand (resp. left-hand) side. Next, change this sign if and only if the w-arrow contains an odd number of twists. 
For example, the Gauss diagram for the right-handed trefoil shown in Figure  \ref{fig:trefoil} is obtained from the Arrow presentation on the right-hand side by labeling all three arrows by $+$. 
Note that, if the head of a w-arrow is attached to the right-hand side of the diagram, then the parity of the number of twists corresponds to the sign. 

Conversely, any Gauss diagram can be converted to an Arrow presentation, 
by attaching the head of an arrow to the right-hand (resp. left-hand) side of the (trivial) diagram if it is labeled by a $+$ (resp. $-$). 

Theorem \ref{thm:main1}  provides a complete calculus (Arrow moves) for this alternative version of Gauss diagrams (Arrow presentations), 
which is to be compared with the Gauss diagram versions of Reidemeister moves.  
Although the set of Arrow moves is larger, and hence less suitable for (say) proving invariance results, it is 
in general much simpler to manipulate. Indeed, Gauss diagram versions of Reidemeister moves III contain rather delicate compatibility conditions, given by both the arrow signs and local orientations of the strands, see \cite{GPV}; Arrow moves, on the other hand, involve no such condition. 

Moreover, we shall see in the next sections that Arrow calculus generalizes widely to w-trees. 
This can thus be seen as an `higher order Gauss diagram' calculus. 

\section{Surgery along w-trees}\label{sec:wtreesurg}

In this section, we show how w-trees allow to generalize surgery along w-arrows. 

\subsection{Subtrees, expansion, and surgery along w-trees}

We start with a couple preliminary definitions.

A \emph{subtree} of a w-tree is a connected union of edges and vertices of this w-tree. \\
Given a subtree $S$ of a $\w$-tree $T$ for a diagram $D$ (possibly $T$ itself), consider for each endpoint $e$ of $S$ a point $e'$ 
on $D$ which is adjacent to $e$, so that $e$ and $e'$ are met consecutively, in this order, when running along $D$ following the orientation. 
One can then form a new subtree $S'$, by joining these new points by the same directed subtree as $S$, so that it runs parallel to it and crosses 
it only at virtual crossings. We then say that $S$ and $S'$ are two \emph{parallel subtrees}. 

We now introduce the \emph{Expansion move} (E), which comes in two versions as shown in Figure \ref{fig:Exp}. 
 \begin{figure}[!h]
  \includegraphics[scale=1]{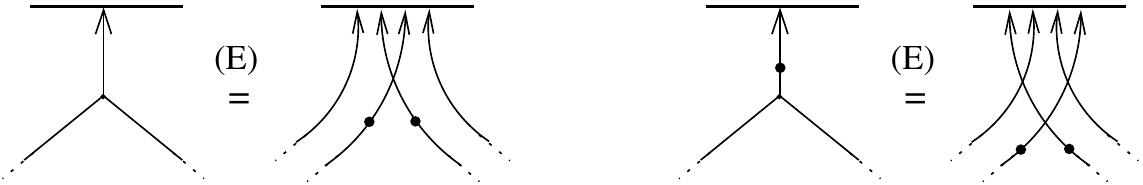}
  \caption{Expanding w-trees by using (E)  }
\label{fig:Exp}
\end{figure}
\begin{convention}\label{conv:parallel}
In Figure \ref{fig:Exp}, the dotted lines on the left-hand side of the equality represent two subtrees, 
forming along with the part which is shown a $\w$-tree.  
The dotted parts on the right-hand side then represent parallel copies of both subtrees. 
Together with the represented part, they form pairs of parallel w-tree which only differ by a twist on the terminal edge. 
See the first equality of Figure \ref{fig:exe} for an example. 
We shall use this diagrammatic convention throughout the paper. 
\end{convention}
By applying (E) recursively,  we can eventually turn any w-tree into a union of w-arrows. 
Note that this process is uniquely defined. 
An example is given in Figure \ref{fig:exe}.
\begin{definition}
The \emph{expansion} of a w-tree is the union of w-arrows obtained from repeated applications of (E). 
\end{definition}
 \begin{figure}[!h]
  \includegraphics[scale=1]{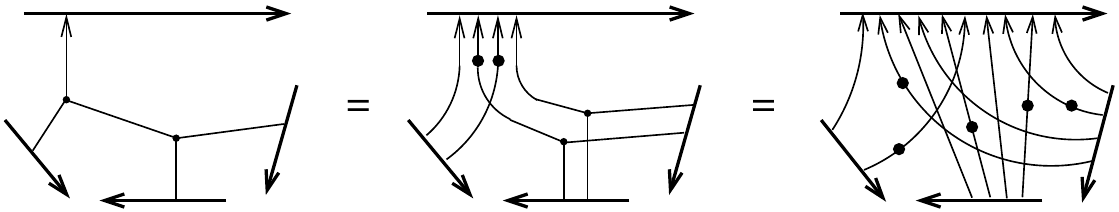}
  \caption{Expansion of a $\w_3$-tree}\label{fig:exe}
\end{figure}

\begin{remark}\label{rem:commutator}
As Figure \ref{fig:exe} illustrates, the expansion of a $\w_k$-tree $T$ takes the form of an `iterated commutators of w-arrows'.  
More precisely, labeling the tails of $T$ from $1$ to $k$, and denoting by $i$ a w-arrow running from (a neighborhood of) tail $i$ to (a neighborhood of) the head of $T$, and by $i^{-1}$ a similar w-arrow with a twist, then the heads of the w-arrows in the expansion of $T$ are met along $D$ according to a $k$-fold commutator in $1,\cdots,k$.  
See Section \ref{sec:algebra} for a more rigorous and detailed treatment.  
\end{remark}

The notion of expansion leads to the following. 
\begin{definition}
The \emph{surgery along a w-tree} is surgery along its expansion. 
\end{definition}
As before, we shall denote by $D_T$ the result of surgery on a diagram $D$ along a union $T$ of w-trees. 

\begin{remark}\label{rem:expansion}
We have the following \emph{Brunnian-type property}. 
Given a w-tree $T$, consider the trivial tangle $D$ given by a neighborhood of its endpoints: the tangle $D_T$ is Brunnian, in the sense that deleting any component yields a trivial tangle. 
Indeed, in the expansion of $T$, we have that deleting all w-arrows which have their tails on a same component of $D$, produces a union of w-arrows which yields a trivial surgery, thanks to the Inverse move (5). 
\end{remark}

\subsection{Moves on w-trees} \label{sec:moves}

In this section, we extend the Arrow calculus set up in Section \ref{sec:warrsurg} to w-trees.
The expansion process, combined with Lemma \ref{lem:wamoves}, gives immediately the following.
\begin{lemma}\label{lem:treemoves}
Arrow moves (1) to (4) hold for w-trees as well. More precisely:
\begin{itemize}
 \item[$\circ$]  one should add the following local moves to (1): 
   \begin{center}
    \includegraphics[scale=1]{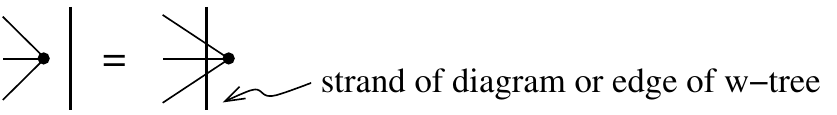}
   \end{center}
 \item[$\circ$]  The Tails Exchange move (3) may involves tails from different components or from a single component. 
\end{itemize}
\end{lemma}
\begin{remark}\label{rem:parallel}
As a consequence of the Tails Exchange move for w-trees, 
the relative position of two (sub)trees for a diagram is completely specified by the relative position of the two heads. 
In particular, we can unambiguously refer to \emph{parallel w-trees} 
by only specifying the relative position of their heads. 
Likewise, we can freely refer to `parallel subtrees' of two w-trees if these subtrees do not contain the head. 
\end{remark}
\begin{convention}
In the rest of the paper, we will use the same terminology for the w-tree versions of moves (1) to (4), and in particular we will use the same numbering. 
As for moves (5) and (6), we will rather refer to the next two lemmas when used for w-trees. 
\end{convention}
As a generalization of the Inverse move (5), we have the following. 
\begin{lemma}[Inverse]\label{lem:inverse}
Two parallel w-trees which only differ by a twist on the terminal edge yield a trivial surgery.\footnote{Recall from Convention \ref{conv:parallel} that, in the figure, the dotted parts represent two parallel subtrees. } \\
\begin{center}
   \includegraphics[scale=1]{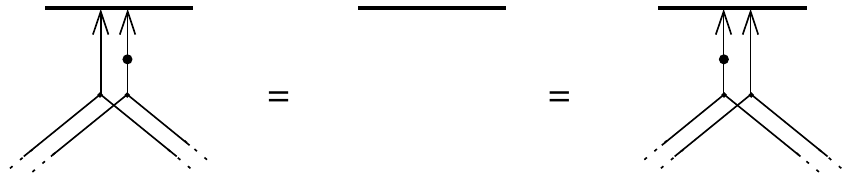}
\end{center}
\end{lemma}
\begin{proof}
We only prove here the first equality, the second one being strictly similar. 
We proceed by induction on the degree of the w-trees involved. 
The w-arrow case is given by move (5). 
Now, suppose that the left-hand side in the above figure involves two $\w_{k}$-trees. 
Then, one can apply (E) to both to obtain a union of eight $\w$-trees of degree $<k$. 
Figure \ref{fig:invproof} then shows how repeated use of the induction hypothesis implies the result.  
 \begin{figure}[!h]
    \reflectbox{\rotatebox[origin=r]{180}{\includegraphics[scale=1]{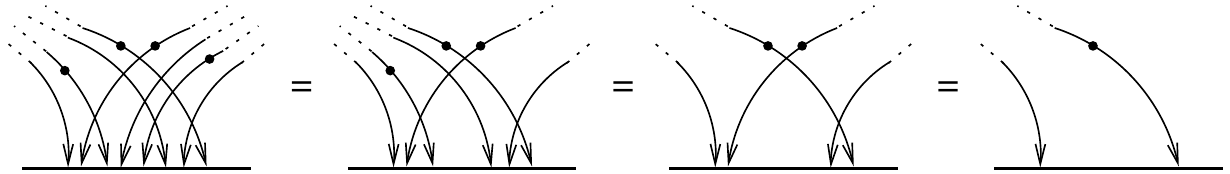}}}
  \caption{Proving the Inverse move for w-trees}\label{fig:invproof}
\end{figure}
\end{proof}

\begin{convention}\label{rem:inverse2}
 In the rest of this paper, when given a union $S$ of w-trees with adjacent heads, we will denote by $\overline{S}$ the union of w-trees such that we have
 \[ \textrm{\includegraphics[scale=0.8]{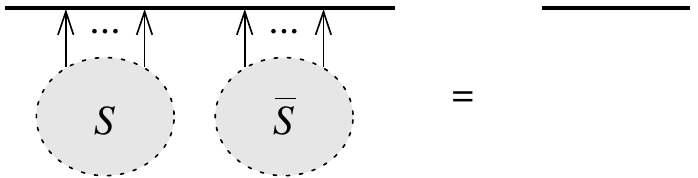}} \]
Note that $\overline{S}$ can be described explicitly from $S$, by using the Inverse Lemma \ref{lem:inverse} recursively. 
We stress that the above graphical convention will always be used for w-trees with adjacent heads, so that no tail is attached to the represented portion of diagram.  
\end{convention}

Likewise, we have the following natural generalization of the Slide move (6). 
\begin{lemma}[Slide]\label{lem:slide}
The following equivalence holds. 
\begin{figure}[!h]
  \includegraphics[scale=1.2]{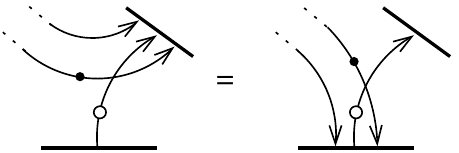}
 \caption{The Slide move for w-trees}\label{fig:sl}
\end{figure}
\end{lemma}
\begin{proof}
The proof is done by induction on the degree of the w-trees involved in the move, as in the proof of Lemma \ref{lem:inverse}. 
The degree $1$ case is the Slide move (6) for w-arrows. 
Now, suppose that the left-hand side in the figure of Lemma \ref{lem:slide} involves two $\w_{k}$-trees, 
and apply (E) to obtain a union of eight $\w$-trees of degree $<k$. 
These $\w$-trees are so that we can slide them pairwise, using the induction hypothesis four times. 
Applying (E) back again to the resulting eight $\w$-trees, we obtain the desired pair of $\w_{k}$-trees. 
\end{proof}
\begin{remark}\label{rem:genslide}
The Slide Lemma \ref{lem:slide} generalizes as follows. 
If one replace the w-arrow in Figure \ref{fig:sl} by a bunch of parallel w-arrows, then the lemma still applies. 
Indeed, it suffices to insert, using the Inverse Lemma \ref{lem:inverse}, 
pairs of parrallel w-trees between the endpoints of each pair of consecutive w-arrows, apply the Slide Lemma \ref{lem:slide}, 
then remove pairwise all the added w-trees again by the Inverse Lemma. 
Note that this applies for any parallel bunch of w-arrows, for any choice of orientation and twist on each individual w-arrow.
\end{remark}

We now provide several supplementary moves for w-trees. 

\begin{lemma}[Head Traversal]\label{lem:jump}
A w-tree head can pass through an isolated union of w-trees:\footnote{In the figure, the shaded part indicates a portion of diagram with some w-trees, which is contained in a disk as shown. } 
\begin{center}
   \includegraphics[scale=1]{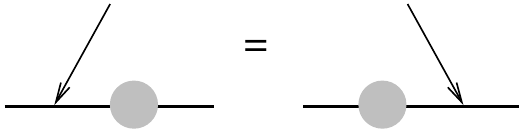}
\end{center}
\end{lemma}
\begin{proof}
Clearly, by (E), it suffices to prove the result for a w-arrow head. 
The proof is given in Figure \ref{fig:jumpproof}.
(More precisely, the figure proves the equality for one choice of orientation; the other case is strictly similar.)
 \begin{figure}[!h]
   \includegraphics[scale=0.95]{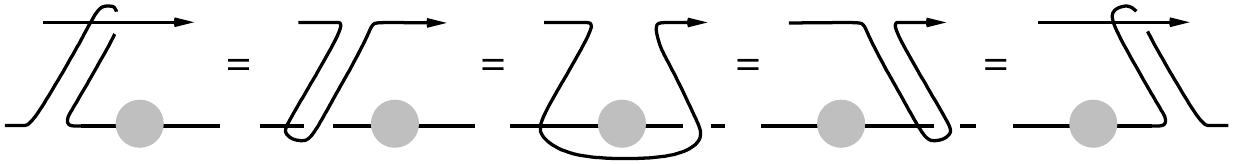}
   \caption{Proving the Head Traversal move}\label{fig:jumpproof}
\end{figure}
Surgery yields the diagram shown on the left-hand side of the figure, which can be deformed into the second diagram by a planar isotopy. 
Successive applications of the detour move and of the w-detour move (Remark \ref{rem:wdetour}) then give the next two equalities, and another planar isotopy completes the proof. 
\end{proof}

\begin{lemma}[Heads Exchange]\label{lem:he}
Exchanging two heads can be achieved at the expense of an additional w-tree, as shown below:  
\begin{center}
   \reflectbox{\rotatebox[origin=c]{180}{\includegraphics[scale=1]{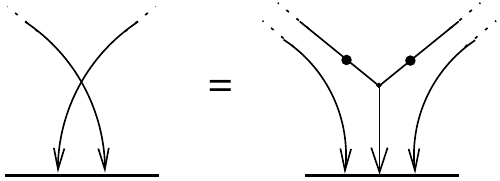}}}
\end{center}
\end{lemma}
\begin{proof}
Starting from the right-hand side of the above equality, applying the Expansion move (E) gives the first equality in Figure \ref{fig:heproof}.  
 \begin{figure}[!h]
    \reflectbox{\rotatebox[origin=c]{180}{\includegraphics[scale=0.9]{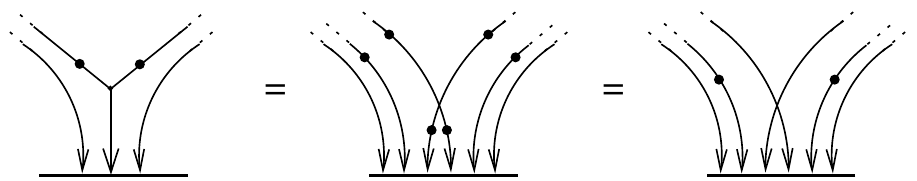}}}
  \caption{Proving the Heads Exchange move}\label{fig:heproof}
\end{figure}
\noindent 
The involutivity of twists gives the second equality, and two applications of 
the Inverse Lemma \ref{lem:inverse} then conclude the proof.    
\end{proof}

\begin{remark}\label{cor:hh}
By strictly similar arguments, one can show the simple variants of the Heads Exchange move given in Figure \ref{fig:head2}.
 \begin{figure}[!h]
   \reflectbox{\rotatebox[origin=c]{180}{\includegraphics[scale=0.9]{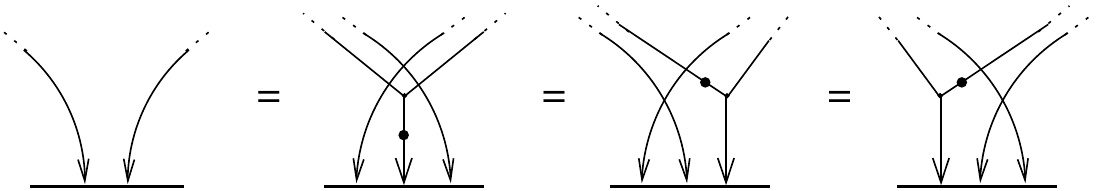}}}
  \caption{Some variants of the Heads Exchange move}\label{fig:head2}
\end{figure}
\end{remark}
\begin{lemma}[Head--Tail Exchange]\label{lem:ht}
Exchanging a w-tree head and a w-arrow tail can be achieved at the expense of an additional w-tree, as shown in Figure \ref{fig:ht}. 
 \begin{figure}[!h]
  \includegraphics[scale=0.95]{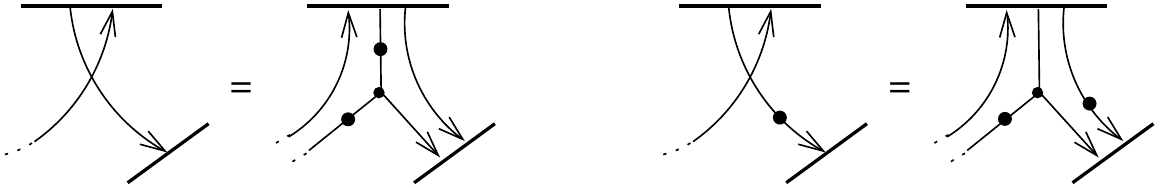}
  \caption{The Head--Tail Exchange move}\label{fig:ht}
 \end{figure}
\end{lemma}

\begin{proof}
We only prove the version of the equality where there is no twist on the left-hand side, the other one being strictly similar. 
The proof is given in Figure \ref{fig:htproof}. 
 \begin{figure}[!h]
   \includegraphics[scale=0.9]{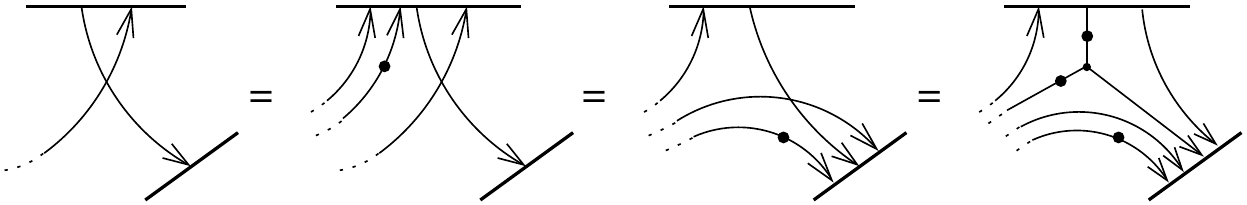}
  \caption{Proving the Head--Tail Exchange move}\label{fig:htproof}
\end{figure}
The three identities depicted there respectively use the Inverse Lemma \ref{lem:inverse}, the Slide Lemma \ref{lem:slide} and the Heads Exchange Lemma \ref{lem:he}. 
Another application of the Inverse Lemma then concludes the argument. 
\end{proof}

\begin{lemma}[Antisymmetry]\label{lem:as}
The cyclic order at a trivalent vertex, induced by the plane orientation, may be changed at the cost of a twist on the three incident edges:  
 \begin{figure}[h]
   \includegraphics[scale=0.9]{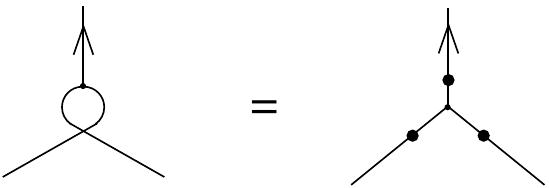}
\end{figure}
\end{lemma}
\begin{proof}
The proof is by induction on the number of edges from the head to the trivalent vertex involved in the move.  
When there is only one edge, the result simply follows from (E), isotopy of the resulting w-trees, and (E) back again, as shown in Figure \ref{fig:asproof1}.
 \begin{figure}[!h]
   \reflectbox{\rotatebox[origin=c]{180}{\includegraphics[scale=1]{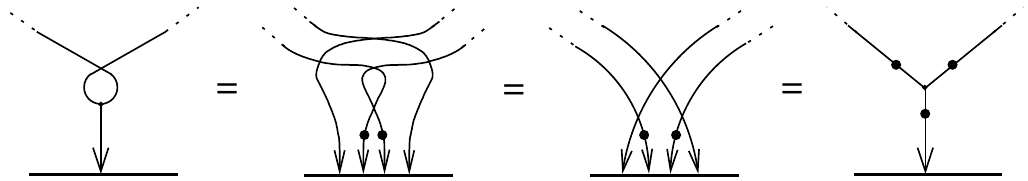}}}
  \caption{Proving the Antisymmetry move 
  }\label{fig:asproof1}
\end{figure}
\noindent (Here, we only show the case where the terminal edge contains no twist: the other case is similar.)
\noindent 
In the general case, we use (E) to apply the induction hypothesis to the resulting w-trees, and use (E) back again, as in the proofs of Lemmas \ref{lem:inverse} and \ref{lem:slide}. 
\end{proof}

A \emph{fork} is a subtree which consists of two adjacent tails  connected to the same trivalent vertex (possibly containing some twists).  
See Figure \ref{fig:fork}. 
\begin{lemma}[Fork move]\label{lem:fork}
Surgery along a w-tree containing a fork does not change the equivalence class of a diagram.
 \begin{figure}[!]
  \reflectbox{\rotatebox[origin=c]{180}{\includegraphics[scale=0.9]{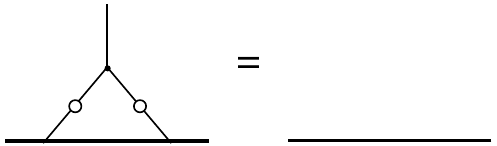}}}
  \caption{The Fork move}\label{fig:fork}
\end{figure}
\end{lemma}
\begin{proof}
The proof is by induction on the number of edges from the head to the fork. 
The initial case of a $\w_2$-tree with adjacent tails is shown in Figure \ref{fig:forkproof}, in the case where no edge contain a twist
(the other cases are similar).  
 \begin{figure}[!h]
  \includegraphics[scale=0.9]{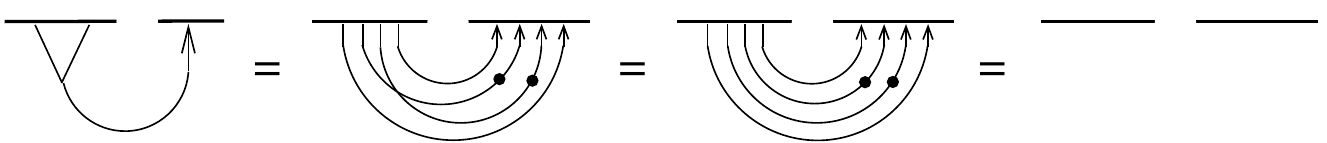}
  \caption{Proving the Fork move}\label{fig:forkproof}
\end{figure}
The inductive step is clear: applying (E) to a w-tree containing a fork yields four w-trees, two of which contain a fork, by the Tails Exchange move (3). 
Using the induction hypothesis, we are thus left with two w-trees which cancel by the Inverse Lemma \ref{lem:inverse}. 
\end{proof}

\subsection{w-tree presentations for welded knotted objects}

We have the following natural generalization of the notion of Arrow presentation. 
\begin{definition}
Suppose that a diagram is obtained from a diagram $U$ \emph{without} classical crossings by surgery along a union $T$ of w-trees.
Then $(U,T)$ is called a \emph{w-tree presentation} of the diagram.\\
Two w-tree presentations are \emph{equivalent} if they represent equivalent diagrams.  
\end{definition}
Let us call \emph{w-tree moves} the set of moves on w-trees given by the results of Section \ref{sec:wtreesurg}. 
More precisely, w-tree moves consists of the Expansion move (E), Moves (1)-(4) of Lemma \ref{lem:treemoves}, and the Inverse (Lem.~ \ref{lem:inverse}), Slide (Lem.~\ref{lem:slide}), Head Traversal (Lem.~\ref{lem:jump}), Heads Exchange (Lem.~\ref{lem:he}), Head--Tail Exchange (Lem.~\ref{lem:ht}), Antisymmetry (Lem.~\ref{lem:as}) and Fork (Lem.~\ref{lem:fork}) moves. 
Clearly, w-tree moves yield equivalent w-tree presentations. 

Examples of w-tree presentations for the right-handed trefoil are given in Figure \ref{fig:trefoil2}. 
There, starting from the Arrow presentation of Figure \ref{fig:trefoil}, we apply the Head--Tail Exchange Lemma \ref{lem:ht}, the Tails Exchange move (3) and the Isolated Arrow move (4). 
\begin{figure}[!h]
   \includegraphics[scale=1]{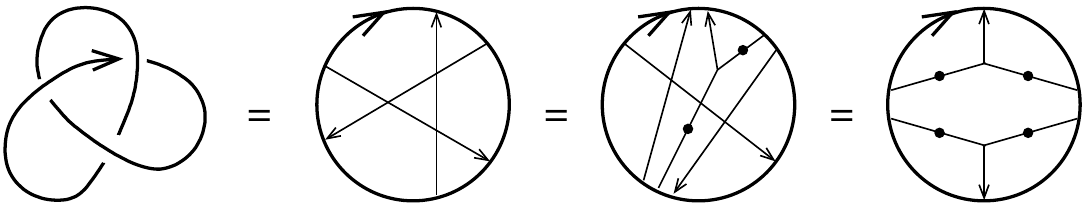}
  \caption{Tree-presentation for the trefoil}\label{fig:trefoil2}
\end{figure}
As mentioned in Section \ref{sec:GD1}, these can be regarded as kinds of `higher order Gauss diagram' presentations for the trefoil. 

\begin{remark}\label{rem:0+0}
As pointed out to the authors by D.~Moussard, Figure \ref{fig:trefoil2} shows that the trefoil can be written as a composite knot when seen as a welded object (note, however, that connected sum is not well-defined for welded knots).   Actually, it follows from the Fork Lemma \ref{lem:fork} that the two factors are equivalent to the unknot, meaning that the trefoil is, rather surprisingly, the composite of two unknots.  In fact, we can show, using bridge presentations and Arrow calculus, that this is the case for \emph{any} $2$-bridge knot. 
\end{remark}

It follows from Theorem \ref{thm:main1} that w-tree moves provide a complete calculus for w-tree presentations. 
In other words, we have the following. 
\begin{theorem}
Two w-tree presentations represent equivalent diagrams if and only if they are related by w-tree moves.
\end{theorem}

Note that the set of w-tree moves is highly non-minimal. 
In fact, the above remains true when only considering the Expansion 
move (E) and Arrow moves (1)-(6). 

\section{Welded invariants}\label{sec:invariants}

In this section, we review several welded extensions of classical invariants. 
\subsection{Virtual knot group}  \label{sec:group}

Let $L$ be a welded (string) link diagram. 

Recall that the \emph{group $G(L)$ of $L$} is defined by a Wirtinger presentation, as follows. 
Each arc of $L$ (i.e. each piece of strand bounded by either a strand endpoint 
or an underpassing arc in a classical crossing) yields a generator, 
and each classical crossing gives a relation, as indicated in Figure \ref{fig:wirtinger}. 
\begin{figure}[!h]
   \includegraphics[scale=0.9]{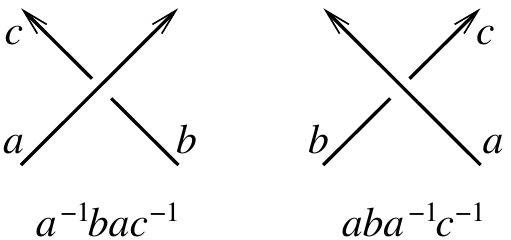}
  \caption{Wirtinger relation at each crossing}\label{fig:wirtinger}
\end{figure}

Since virtual crossings do not produce any generator or relation, virtual and Mixed Reidemeister moves obviously preserve the group presentation \cite{Kauffman}.
It turns out that this `virtual knot group' is 
also invariant under the OC move, and is thus a welded invariant \cite{Kauffman,Satoh}. 

\subsubsection{Wirtinger presentation using w-trees}\label{sec:wirtinger}

Given a w-tree presentation of a diagram $L$, we can associate a Wirtinger presentation of $G(L)$ which involves in general fewer generators and relations. 
More precisely, let $(U,T)$ be a w-tree presentation of $L$, where $T=T_1\cup \cdots \cup T_r$ has $r$ connected components. 
The $r$ heads of $T$ split $U$ into a collection of $n$ arcs,\footnote{More precisely, the heads of $T$ split $U$ into a collections of arcs and possibly several circles, corresponding to closed components of $U$ with no head attached. }
and we pick a generator $m_i$ for each of them. 
Consider the free group $F$ generated by these generators, where the inverse of a generator  $m_i$ will be denoted by $\ov{m_i}$.
Arrange the heads of $T$ (applying the Head Reversal move (2) if needed) so that it looks locally as in Figure \ref{fig:wirt2}.
Then we have 
$$ G(L) = \langle \{m_i\}_i \,\vert \, R_j\, (j=1,\cdots ,r)  \rangle, $$
where $R_j$ is a relation associated with $T_j$ as illustrated in the figure. 
There, $w(T_j)$ is a word in $F$, constructed as follows. 
\begin{figure}[!h]
   \includegraphics[scale=0.9]{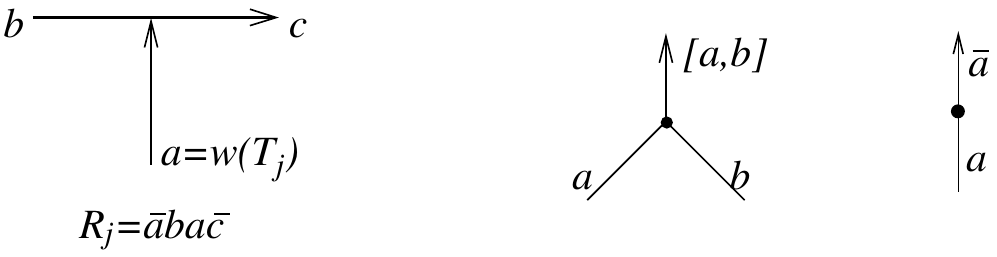}
  \caption{Wirtinger-type relation at a head, and the procedure to define $w(T)$}\label{fig:wirt2}
\end{figure}

First, label each edges of $T_j$ which is incident to a tail by the generator $m_i$ inherited from its attaching point. 
Next, label all edges of $T_j$ by elements of $F$ by applying recursively the rules illustrated in Figure \ref{fig:wirt2}. 
More precisely, assign recursively to each outgoing edge at a trivalent vertex the formal bracket 
$$[a,b]:=a\ov{b}\ov{a}b, $$
where $a$ and $b$ are the labels of the two ingoing edge, following the plane orientation around the vertex; we also require that a label meeting a twist is replaced by its inverse. This procedure yields a word $w(T_j)\in F$ associated to $T_j$, which is defined as the label at its terminal edge. 
Note that this procedure more generally associates a formal word to any subtree of $T_j$, and that, by the Tail Reversal move (2), the local orientation of the diagram at each tail is not relevant in this process.  

In the case of a canonical Arrow presentation of a diagram, the above procedure recovers the usual Wirtinger presentation of the diagram, 
and it is easily checked that, in general, this procedure indeed gives a presentation of the same group.  

\begin{remark}\label{rem:faithful}
 As outlined in Section \ref{sec:ribbon}, the Tube map that `inflates' a welded diagram $L$ into a ribbon knotted surface acts faithfully on the virtual knot group, in the sense that we have an isomorphism $G(L)\cong \pi_1\big(\Tube(L)\big)$,\footnote{Here, $\pi_1\big(\Tube(L)\big)$ denote the fundamental group of the complement of the surface  $\Tube(L)$ in $4$-space. } which maps meridians to meridians and (preferred) longitudes to (preferred) longitudes, so that the Wirtinger presentations are in one--to--one correspondence; see \cite{Satoh, yajima,ABMW}. 
\end{remark}

\subsubsection{Algebraic formalism for w-trees}\label{sec:algebra}

Let us push a bit further the algebraic tool introduced in the previous section.  

Given two w-trees $T$ and $T'$ with adjacent heads in a w-tree presentation, such that the head of $T$ is met before that of $T'$ when following the orientation, we define 
 $$ w(T\cup T'):=w(T)w(T')\in F. $$
\begin{convention}
Here $F$ denotes the free group on the set of Wirtinger generators of the given w-tree presentation, as defined in Section \ref{sec:wirtinger}. 
In what follows, we will always use this implicit notation. 
\end{convention}

Note that, if $\ov{T}$ is obtained from $T$ by inserting a twist in its terminal edge, then $w(\ov{T})=\ov{w(T)}$, and $w(T\cup \ov{T})=1$, which is compatible with Convention \ref{rem:inverse2}. 

Now, if we denote by $E(T)$ the result of one application of (E) to some w-tree $T$, then we have $w(T)=w(E(T))$. 
More precisely, if we simply denote by $A$ and $B$ the words associated with the two subtrees at the two ingoing edges of the vertex where (E) is applied, then we have 
$$ w(T)=[A,B]=A\, \ov{B}\, \ov{A}\, B. $$

We can therefore reformulate (and actually, easily reprove) some of the results of Section \ref{sec:moves} in these algebraic terms. 
For example, the Heads Exchange Lemma \ref{lem:he} translates to 
$$ AB = B[\ov{B},\ov{A}]A, $$
and its variants given in Figure \ref{fig:head2}, to 
$$ AB = B\ov{[A,B]}A = BA[\ov{A},B] = [A,\ov{B}]BA. $$
The Antisymmetry Lemma \ref{lem:as} also reformulates nicely; for example the `initial case' shown in Figure \ref{fig:asproof1} can be restated as 
$$ [B,A] = \ov{[\ov{A},\ov{B}]}. $$
Finally, the Fork Lemma \ref{lem:fork} is simply
$$ [ \mydot [A,A] \mydot ]=1. $$

In the sequel, although we will still favor the more explicit diagrammatical language, we shall sometimes make use of this algebraic formalism.  

\subsection{The normalized Alexander polynomial for welded long knots}  \label{sec:alex}

Let $L$ be a welded long knot diagram.
Suppose that the group of $L$ has presentation 
$
G(L) = \langle x_1,\cdots, x_{m} \vert r_1,\cdots, r_{n} \rangle$ for some $m,n$. 
Consider the $n\times m$ \emph{Jacobian matrix} 
$M=\left( \varphi\left(\dfrac{\partial r_i}{\partial x_j}\right) \right)_{i,j}$,  
where $\frac{\partial }{\partial x_j}$ denote the Fox free derivative in variable $x_j$, and where 
$\varphi: \mathbb{Z} F(x_1,\cdots, x_{m})\rightarrow \mathbb{Z}[t^{\pm 1}]$ 
is the ring homomorphism mapping each generator $x_i$ of the free group $F(x_1,\cdots, x_{m})$ to $t$.
 
The \emph{Alexander polynomial} of $L$, denoted by $\Delta_L(t)\in \mathbb{Z}[t^{\pm 1}]$, 
is defined as the greatest common divisor of the 
$(m-1)\times (m-1)$-minors of $M$, which is well-defined up to a unit factor. 

In order to remove the indeterminacy in the definition of $\Delta_L(t)$, we further require that $\Delta_L(1)=1$ and that $\frac{d\Delta_L}{dt}(1)=0$. 
The resulting invariant is the \emph{normalized Alexander polynomial} of $L$, denoted by $\tilde\Delta_L$ (see e.g. \cite{HKS}).
Taking the power series expansion at $t=1$ as 
$$
\tilde\Delta_L(t) = 1 + \sum_{k\ge 2} \alpha_k(L) (1-t)^k 
$$
thus defines an infinite sequence of integer-valued invariants $\alpha_k$ of welded long knots. 
(Our definition slightly differs from the one used in \cite{HKS}, by a factor $(-1)^k$.) 
\begin{definition}
We call the invariant $\alpha_k$ the $k$th \emph{normalized coefficient} of the Alexander polynomial. 
\end{definition}
We now give a realization result for the coefficients $\alpha_k$ in terms of w-trees. 
Consider the welded long knots $L_k$ or $\overline{L_k}$ ($k\ge 2$) defined in Figure \ref{fig:K0}. 
\begin{figure}[!h]
   \includegraphics[scale=0.9]{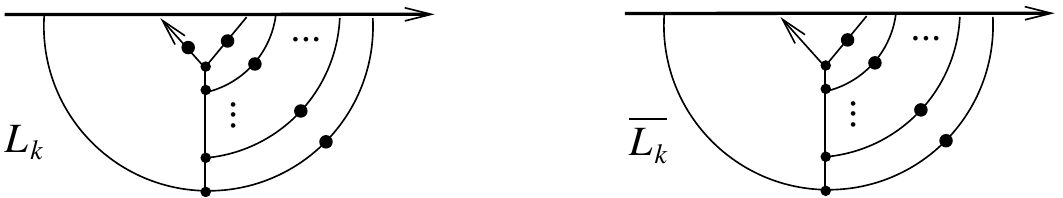}
  \caption{The welded long knots $L_k$ or $\overline{L_k}$, given by surgery along a single $\w_k$-tree ($k\ge 2$).}\label{fig:K0}
\end{figure}
\begin{lemma} \label{lem:wkAlex}
Let $k\ge 2$. The normalized Alexander polynomial of $L_k$ and $\overline{L_k}$ are given by 
 $$ \tilde\Delta_{L_k}(t) = 1 + (1-t)^{k} \quad\textrm{and}\quad \tilde\Delta_{\overline{L_k}}(t) = 1 - (1-t)^{k}. $$
\end{lemma}
Note that these are genuine equalities: there are no higher order terms. 
In particular, we have $\alpha_i(L_k) = -\alpha_i(\overline{L_k})= \delta_{ik}$. 
\begin{proof}[Proof of Lemma \ref{lem:wkAlex}]
 The presentation for $G(L_k)$ given by the defining $\w_k$-tree presentation is 
$\langle l,r  \vert  R_k l R_k^{-1} r^{-1} \rangle$, 
 where  $R_k=\big[ [\cdots[[[l,r^{-1}],r^{-1}],r^{-1}]\cdots ],r^{-1}\big]$ is a length $k$ commutator.
 One can show inductively that 
 \[\textrm{$\varphi\left(\dfrac{\partial R_k}{\partial l}\right) =   (1-t)^{k-1}\,$  and 
 $\,\varphi\left(\dfrac{\partial R_k}{\partial r}\right) =   -(1-t)^{k-1}$,}\] 
 so that the normalized Alexander polynomial is given by 
 $ \tilde\Delta_{L_k}(t)  = 1 + (1-t)^{k}$. 
The result for $\overline{L_k}$ is completely similar, and is left to the reader. 
\end{proof}

The following might be well-known; the proof is completely straightforward and is thus omitted.
\begin{lemma}\label{lem:additive}
 The normalized Alexander polynomial of welded long knots is multiplicative. 
\end{lemma}
Lemma~\ref{lem:additive} implies the following additivity result.
\begin{corollary}\label{cor:additive}
Let $k$ be a positive integer and let $K$ be a welded long knot with $\alpha_i(K)=0~(i\leq k-1)$. 
Then, for any welded long knot $K'$, 
$\alpha_k(K\cdot K') = \alpha_k(K) + \alpha_k(K')$.
\end{corollary}

\subsection{Welded Milnor invariants}  \label{sec:milnor}

We now recall the general virtual extension of Milnor invariants given in \cite{ABMW}, which is an invariant of welded string links. 
This construction is intrasically topological, since it is defined via the Tube map as the $4$-dimensional analogue of Milnor invariants 
for (ribbon) knotted annuli in $4$-space; we will however give here a purely combinatorial reformulation. 

Given an $n$-component welded string link $L$, consider the group $G(L)$ defined in Section \ref{sec:group}. 
Consider also the free group $F^l$ and $F^u$ generated by the $n$ `lower' and `upper' Wirtinger generators, i.e. the generators  associated with the $n$ arcs of $L$ containing the initial, resp. terminal, point of each component.  
Recall that the lower central series of a group $G$ is the family of nested subgroups $\{\Gamma_kG\}_{k\ge 1}$ defined recursively by 
$\Gamma_1 G=G$ and $\Gamma_{k+1} G=[G,\Gamma_k G]$. 
Then, for each $k\ge 1$, we have a sequence of isomorphisms\footnote{This relies heavily on the topological realization of 
welded string links as ribbon knotted annuli in $4$-space by the Tube map, which acts faithfully at the level of the group system: see Section 5 of \cite{ABMW} for the details.  }
 $$ F_n/\Gamma_k F_n \simeq F^l/\Gamma_k F^l \simeq G(L)/\Gamma_k G(L)\simeq F^u/\Gamma_k F^u\simeq F_n/\Gamma_k F_n, $$
where $F_n$ is the free group on $m_1,\cdots,m_n$. 
In this way, we associate to $L$ an element $\varphi_k(L)$ of $\Aut(F_n/\Gamma_k F_n)$. 
This is more precisely a conjugating automorphism, in the sense that, 
for each $i$, $\varphi_k(L)$ maps $m_i$ to a conjugate $m_i^{\lambda^k_i}$; 
we call this conjugating element $\lambda^k_i\in F_n/\Gamma_k F_n$ the \emph{combinatorial $i$th longitude}. 
Now, consider the \emph{Magnus expansion}, which is the group homomorphism $E:F_n\rightarrow \mathbb{Z}\langle\langle X_1,\cdots,X_n\rangle\rangle$ 
mapping each generator $m_i$ to the formal power series $1+X_i$.  
\begin{definition}
 For each sequence $I=i_1\cdots i_{m-1}i_m$ of (possibly repeating) indices in $\{1,\cdots,n\}$, 
 the \emph{welded Milnor invariant} $\mu^w_I(L)$ of $L$ is the coefficient of the monomial $X_{i_1}\cdots X_{i_{m-1}}$ in $E(\lambda_{i_m}^k)$, for any $k\ge m$. The number of indices in $I$ is called the \emph{length} of the invariant.
\end{definition}

For example, the simplest welded Milnor invariants $\mu^w_{ij}$ indexed by two distinct integers $i,j$ 
are the so-called virtual linking numbers $lk_{i/j}$ (see \cite[\S 1.7]{GPV}. 

\begin{remark}\label{rem:classico}
 This is a welded extension of the classical Milnor $\mu$-invariants, in the sense that if $L$ is a (classical) string link, then 
 $\mu_I(L)=\mu^w_I(L)$ for any sequence $I$. 
\end{remark}

The following realization result, in terms of w-trees, is to be compared with \cite[pp.190]{Milnor} and \cite[Lem.~4.1]{yasuhara}. 
\begin{lemma}\label{lem:Milnor}
Let $I=i_1\cdots i_k$ be a sequence of indices in $\{1,\cdots,n\}$, and, for any $\sigma$ in the symmetric group $S_{k-2}$ of degree $k-2$, set
$\sigma(I)=i_{\sigma(1)}\cdots i_{\sigma(k-2)} i_{k-1}i_k$. 
\begin{figure}[!h]
   \includegraphics[scale=0.8]{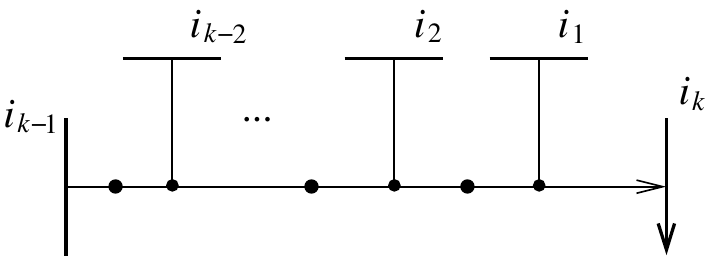}
  \caption{The $\w_{k-1}$-tree $T_{i_1,i_2,\cdots ,i_{k-1},i_k}$ for $\1_n$}\label{fig:wMilnor}
\end{figure}
Consider the w-tree $T_I$ for the trivial $n$-string link diagram $\1_n$ shown in Figure \ref{fig:wMilnor}. 
Then we have 
$$ \mu^w_{\sigma(I)}\left((\1_n)_{T_I}\right)  
= \left\{ \begin{array}{ll}
					    1 & \textrm{if $\sigma=$id,} \\
					    0 & \textrm{otherwise.} \\
					    \end{array}\right.
$$
Moreover for all $\sigma\in S_{k-2}$, we have  
$$\mu^w_{\sigma(I)}\left((\1_n)_{T_I}\right) = -\mu^w_{\sigma(I)}\left((\1_n)_{\overline{T}_I}\right), $$ 
where $\overline{T}_I$ is the w-tree obtained from $T_I$ by inserting a twist in the terminal edge.
\end{lemma}
\begin{proof}
This is a straightforward calculation, based on the observation that the combinatorial $i_k$th longitude of $T_I$ is given by  
$$\lambda^k_{i_k}=[i_1,[i_2,\cdots ,[i_{k-3},[i_{k-2},i_{k-1}^{-1}]^{-1}]^{-1}\cdots]^{-1} ]$$ (all other longitudes are clearly trivial). 
\end{proof}

\begin{remark}
The above definition can be adapted to welded link invariants, which involves, as in the classical case, a recurring indeterminacy depending on lower order invariants. In particular, the first non-vanishing invariants are well defined integers, and Lemma \ref{lem:Milnor} applies in this case. 
\end{remark}

Finally, let us add the following additivity result. 
\begin{lemma}\label{lem:Madditive}
Let $L$ and $L'$ be two welded string links of the same number of components.
Let $m$, resp. $m'$, be the integer such that all welded Milnor invariants of $L$, resp. $L'$, of length $\le m$, resp. $\le m'$, are zero. 
Then $\mu^w_I(L\cdot L') = \mu^w_I(L) + \mu^w_I(L')$ for any sequence $I$ of length $\le m+m'$.
\end{lemma}
\noindent The proof is strictly the same as in the classical case, as for example in \cite[Lem. 3.3]{MY1}, and is therefore left to the reader. 

\subsection{Finite type invariants} \label{sec:fti}

The \emph{virtualization move} is a local move on diagrams which replaces a classical crossing by a virtual one. 
We call the converse local move the \emph{devirtualization} move. 

Given a welded diagram $L$, and a set $C$ of classical crossings of $L$, we denote by $L_C$ the welded diagram obtained by applying the virtualization move to all crossings in $C$; we also denote by $\vert C\vert$ the cardinality of $C$. 

\begin{definition}[\cite{GPV}]
 An invariant $v$ of welded knotted objects, taking values in an abelian group, is a \emph{finite type invariant of degree $\le k$} if, 
 for any welded diagram $L$ and any set $S$ of $k+1$ classical crossings  of $L$, we have 
 \begin{equation}\label{eq:fti}
   \sum_{S'\subset S} (-1)^{\vert S'\vert } v\left( L_{S'}\right) = 0. 
 \end{equation}
\end{definition}
An invariant is of degree $k$ if it is of degree $\le k$, but not of degree $\le k-1$.

\begin{remark}
This definition is strictly similar to the usual notion of finite type (or Goussarov-Vassiliev) invariants for classical knotted objects, with the virtualization move now playing the role of the crossing change. 
Since a crossing change can be realized by (de)virtualization moves, we have that the restriction of any welded finite type invariant to classical objects is a Goussarov-Vassilev invariants.
\end{remark}

The following is shown in \cite{HKS} (in the context of ribbon $2$-knots, see Remark \ref{rem:faithful}).
\begin{lemma}\label{lem:Alexfti}
For each $k\ge 2$, the $k$th normalized coefficient $\alpha_k$ of the Alexander polynomial is a finite type invariant of degree $k$. 
\end{lemma}

It is known that classical Milnor invariants are of finite type \cite{BNh,Lin}. 
Using essentially the same arguments, it can be shown that, for each $k\ge 1$, length $k+1$ welded Milnor invariants of string links are finite type invariants of degree $k$. 
The key point here is that a virtualization, just a like a crossing change, corresponds to conjugating or not at the virtual knot group level. 
Since we will not make use of this fact in this paper, we will not provide a full and rigorous proof here. 
Indeed, formalizing the above very simple idea, as done by D.~Bar-Natan in \cite{BNh} in the classical case, turns out to be rather involved. 
Note, however, that we will use a consequence of this fact which, fortunately, can  easily be proved directly, see Remark \ref{rem:Milnorwk}.

\begin{remark}\label{rem:ftitube}
The Tube map recalled in Section \ref{sec:ribbon} is also compatible with this finite type invariant theory, 
in the following sense. 
Suppose that some invariant of welded knotted objects $v$ extends naturally to an invariant $v^{(4)}$ of ribbon knotted objects, so that 
$$ v^{(4)}(\Tube(D))=v(D), $$
for any diagram $D$. 
Note that this is the case for the virtual knot group, the normalized Alexander polynomial and welded Milnor invariants, essentially by Remark \ref{rem:faithful}. 
Then, if $v$ is a degree $k$ finite type invariant, then so is $v^{(4)}$, in the sense of the finite type invariant theory of \cite{HKS,KS}. 
Indeed, if two diagrams differ by a virtualization move, then their images by Tube differ by a `crossing changes at crossing circles', 
which is a local move that generates the finite type filtration for ribbon knotted objects, see \cite{KS}.  
\end{remark}

\section{$\w_k$-equivalence}\label{sec:wkequiv}

We now define and study a family of equivalence relations on welded knotted objects, using w-trees. 
We explain the relation with finite type invariants, and give several supplementary technical lemmas for w-trees. 

\subsection{Definitions}\label{sec:wk}

\begin{definition}
 For each $k\ge 1$, the $\w_k$-equivalence is the equivalence relation on welded knotted objects generated by generalized Reidemeister  moves and surgery along $\w_l$-trees, $l\ge k$.
More precisely, two welded knotted objects $W$ and $W'$ are $\w_k$-equivalent if there exists a finite sequence $\{W^i\}_{i=0}^n$ 
 of welded knotted objects such that, for each $i\in\{1,\cdots,n\}$, $W^i$ is obtained from $W^{i-1}$ either by a generalized Reidemeister move or by surgery along a $\w_l$-tree,  for some $l\ge k$.
\end{definition}

By definition, the  $\w_k$-equivalence becomes finer as the degree $k$ increases, in the sense that the 
$\w_{k+1}$-equivalence implies the $\w_k$-equivalence. 

The notion of $\w_k$-equivalence is a bit subtle, in the sense that it involves both moves on diagrams and on w-tree presentations. 
Let us try to clarify this point by introducing the following.
\begin{notation}\label{not:wk}
 Let $(V,T)$ and $(V,T')$ be two w-tree presentations of some diagrams, and let $k\ge 1$ be an integer.
 Then we use the notation 
  $$ (V,T) \lr{k} (V,T') $$ 
 if there is a union $T''$ of w-trees for $V$ of degree $\ge k$ such that $(V,T)=(V,T'\cup T'')$.
 \end{notation}

Note that we have the implication
$$  \left((V,T) \lr{k} (V,T')\right) \, \Rightarrow \, \left(V_T \stackrel{k}{\sim} V_{T'}\right). $$ 
Therefore, statements given in the terms of Notation \ref{not:wk} will be given when possible. 

The converse implication, however, does not seem to hold in general. 
In other words, we do not know whether a $\w_k$--equivalence version of \cite[Prop.~3.22]{Habiro} holds.  

\subsection{Cases $k=1$ and $2$}\label{sec:w1}

We now observe that $\w_1$-moves and $\w_2$-moves are equivalent to familiar local moves on diagrams. 

We already saw in Figure \ref{fig:wcross} that surgery along a w-arrow is equivalent to a devirtualization move. 
Clearly, by the Inverse move (5), this is also true for a virtualization move. 
It follows immediately that any two welded links or string links of the same number of components are $\w_1$-equivalent. 

Let us now turn to the $\w_2$-equivalence relation. 
Recall that the right-hand side of Figure \ref{fig:moves} depicts the \emph{UC move}, 
which is the forbidden move in welded theory. 
We have 
\begin{lemma}
A $\w_2$-move is equivalent to a UC move. 
\end{lemma}
\begin{proof}
Figure \ref{fig:uc} below shows that the UC move is realized by surgery along a $\w_2$-tree. 
Note that, in the figure, we had to choose several local orientations on the strands: we leave it to the reader to check that the other cases of local orientations 
follow from the same argument, by simply inserting twists near the corresponding tails.  
\begin{figure}[!h]
   \includegraphics[scale=1.05]{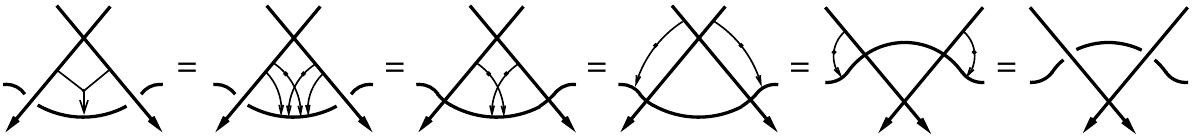}
  \caption{Surgery along a $\w_2$-tree implies the UC move}\label{fig:uc}
\end{figure}

Conversely, Figure \ref{fig:uc2} shows that surgery along a $\w_2$-tree is achieved by the UC move, hence that these two local moves are equivalent. 
\begin{figure}[!h]
   \includegraphics[scale=0.9]{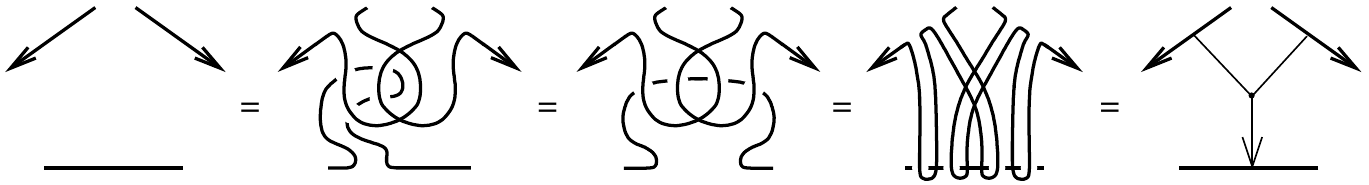}
  \caption{The UC move implies surgery along a $\w_2$-tree}\label{fig:uc2}
\end{figure}
\end{proof}
It was shown in \cite{ABMW3} that two welded (string) links are related by a sequence of UC move, i.e. are $\w_2$-equivalent, if and only if they have same welded Milnor invariants $\mu^w_{ij}$.  
In particular, any two welded (long) knots are $\w_2$-equivalent. 

\begin{remark}
The fact that any two welded (long) knots are $\w_2$-equivalent can also easily be checked directly using arrow calculus. 
Starting from an Arrow presentation of a welded (long) knot, one can use the (Tails, Heads and Head--Tail) Exchange 
move (3) and Lemmas \ref{lem:he} and \ref{lem:ht} to separate and isolate all w-arrows, as in the figure of the Isolated move (4), up to addition of higher order w-trees. 
Each w-arrow is then equivalent to the empty one by move (4). 
\end{remark}

\subsection{Relation to finite type invariants} \label{sec:noteworthy}

One of the main point in studying welded (and classical) knotted objects up to $\w_k$-equivalence is the following.
\begin{proposition}\label{prop:ftiwk}
Two welded knotted objects that are $w_k$-equivalent ($k\ge 1$) cannot be distinguished by finite type invariants of degree $<k$. 
\end{proposition}
\begin{proof}
The proof is formally the same as Habiro's result relating $C_n$-equivalence (see Section \ref{sec:cneq}) to Goussarov-Vassiliev finite type invariants \cite[\S 6.2]{Habiro}, and is summarized below.  

First, recall that, given a diagram $L$ and $k$ unions $W_1,\cdots,W_k$ of 
w-arrows for $L$, the bracket $[L;W_1,\cdots,W_k]$ stands for the formal linear combination of diagrams
 $$ [L;W_1,\cdots,W_k] := \sum_{I\subset\{1,\cdots,k\}} (-1)^{\vert I\vert} L_{\cup_{i\in I} W_i}. $$

Note that, if each $W_i$ consists of a single w-arrow, then the defining equation (\ref{eq:fti}) of finite type invariants can be reformulated 
as the vanishing of (the natural linear extension of) a welded invariant on such a bracket.  
Note also that if, say, $W_1$ 
is a union of w-arrow $W_1^1,\cdots,W_1^{n}$, then we have the equality
 $$ [L;W_1,W_2,\cdots,W_k] = \sum_{j=1}^{n} [L_{W_1^1\cup\cdots \cup W_1^{j-1}};W_1^j,W_2,\cdots,W_k]. $$
Hence an invariant is of degree $<k$ vanishes on $[L;W_1,\cdots,W_k]$. 

Now, suppose that $T$ is a $\w_k$-tree for some diagram $L$, and label the tails of $T$ from $1$ to $k$. 
Consider the expansion of $T$, and denote by $W_i$ the union of all w-arrows running from (a neighborhood of) tail $i$ to (a neighborhood of) the head of $T$. Then $L_T=L_{\cup_{i=1}^k W_i}$ and, according to the Brunnian-type property of w-trees noted in Remark \ref{rem:expansion}, we have $L_{\cup_{i\in I} W_i}=L$ for any $I\subsetneq \{1,\cdots,k\}$. 
Therefore, we have 
$$ L_T-L= (-1)^k[L;W_1,\cdots ,W_k], $$
which, according to the above observation, implies Proposition \ref{prop:ftiwk}.
\end{proof}

We will show in Section \ref{sec:wk_knots} that the converse of Proposition \ref{prop:ftiwk} holds for welded knots and long knots. 

\begin{remark}\label{rem:Milnorwk}
It follows in particular from Proposition \ref{prop:ftiwk} that Milnor invariants of length $\le k$ are invariants under $w_k$-equivalence. 
This can also be shown directly by noting that, if we perform surgery on a diagram $L$ along some $\w_k$-tree, this can only change elements of $G(L)$ by terms in $\Gamma_k F$. 
\end{remark}

\subsection{Some technical lemmas}\label{sec:tech}

We now collect some supplementary technical lemmas, in terms of $\w_k$-equivalence. 
The next result allows to move twists across vertices. 
\begin{lemma}[Twist]\label{lem:twist}
Let $k\ge 2$. The following holds for a $\w_{k}$-tree. 
 \[ \textrm{\includegraphics[scale=1]{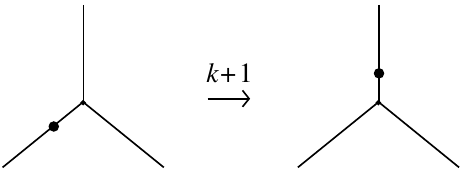}} \]
\end{lemma}
\noindent Note that this move implies the converse one, by using the Antisymmetry Lemma \ref{lem:as} and the twist involutivity.

\begin{proof}
Denote by $n$ the number of edges of $T$ in the unique path connecting the trivalent vertex shown in the statement to the head. 
 Note that $1\le n\le k-1$. 
 We will prove by induction on $d$ the following claim, which is a stronger form of the desired statement. 
 \begin{claim} \label{claim:inverse}
  For all $k\ge 2$, and for any $\w_{k}$-tree $T$, the following equalities hold 
   \[ \textrm{\includegraphics[scale=1]{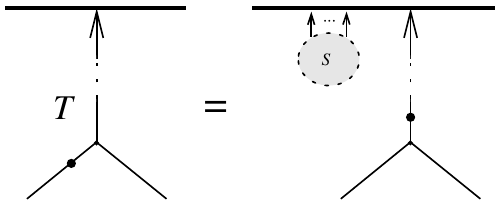}} \]
  \noindent where $S$ is a union of $\w$-trees of degree $>k$ (using the graphical Convention \ref{rem:inverse2}). 
 \end{claim}
  The case $n=1$ of the claim is given in Figure \ref{fig:twistproof}. There, the first equality uses (E) and  the second equality follows from the 
 Heads Exchange Lemma \ref{lem:he} (actually, from  Remark \ref{cor:hh}) applied at the two rightmost heads, and the Inverse Lemma \ref{lem:inverse}. 
 The third equality also follows from Remark \ref{cor:hh} and the Inverse Lemma. 
   \begin{figure}[!h]
   \includegraphics[scale=1]{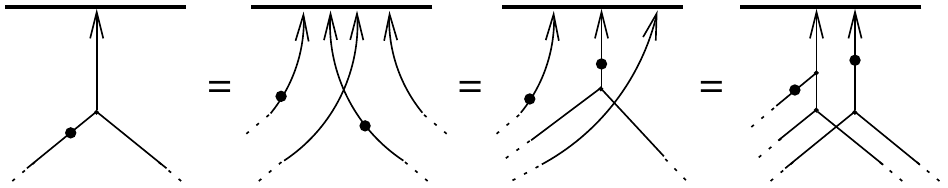} 
   \caption{Proof of the Twist Lemma: case $n=1$}\label{fig:twistproof}
   \end{figure}
 \noindent Note that this can equivalently 
 be shown using the algebraic formalism of Section \ref{sec:algebra}; more precisely, 
 the above figure translates to the simple equalities 
 $$ [\ov{A},B] = \ov{A}\, \ov{B}\, AB\,  = \ov{A}\, \ov{[A,B]}\, A = [\ov{A},[A,B]]\, \ov{[A,B]}. $$
 
 Observe that, in this algebraic setting, $d=n-1$ is the \emph{depth} of $[\ov{A},B]$ in an iterated commutator $[\mydot ,[\ov{A},B]\mydots ]\in \Gamma_{m} F$, 
 which is defined  as the number of elements $D_i\in F$ such that $[\mydot ,[\ov{A},B]\mydots ]=[D_d,[D_{d-1}, \mydot ,[D_1,[\ov{A},B]]\mydots ]]$. 
 For the inductive step, consider an element 
 $\left[ C,[\mydot ,[\ov{A},B]\mydots ]\right]\in \Gamma_k F$, where $C\in \Gamma_l F$ and  
 $[\mydot ,[\ov{A},B]\mydots ]\in \Gamma_m F$ for some integers $l,m$ such that $l+m=k$. 
Observe also that the induction hypothesis gives the existence of some $S\in \Gamma_{m+1} F$ such that 
 $$ [\mydot ,[\ov{A},B]\mydots ]= S\, [\mydot ,\ov{[A,B]}\mydots]. $$
 The inductive step is then given by 
 \begin{align*}
 \left[ C,[\mydot ,[\ov{A},B]\mydots ]\right] & =  C\, \ov{[\mydot ,[\ov{A},B]\mydots]}\, \ov{C}\, [\mydot ,[\ov{A},B]\mydots] & \\
  & =  C\, \ov{[\mydot ,\ov{[A,B]}\mydots]}\, \ov{S}\, \ov{C}\, S\, [\mydot ,\ov{[A,B]}\mydots] & \textrm{ (induction hypothesis)} \\
  & =  G\, C\, \ov{[\mydot ,\ov{[A,B]}\mydots]}\, \ov{C}\, [\mydot ,\ov{[A,B]}\mydots] & \textrm{ (Heads Exch. Lem.~\ref{lem:he})} \\ 
  & =  G\, \left[ C,[\mydot,\ov{[A,B]}\mydots ]\right], 
 \end{align*}
 where $G$ is some term in $\Gamma_{k+1} F$. 
 (The reader is invited to draw the corresponding diagrammatic argument.)
\end{proof}

\begin{remark}
By a symmetric argument, we can prove a variant of Claim \ref{claim:inverse} where the heads of $S$ are to the right-hand side of the w-tree in the figure. 
\end{remark}

Next, we address the move exchanging a head and a tail of two w-trees of arbitrary degree.  
\begin{lemma}\label{lem:th}
The following holds. 
 \[ \textrm{\includegraphics[scale=1]{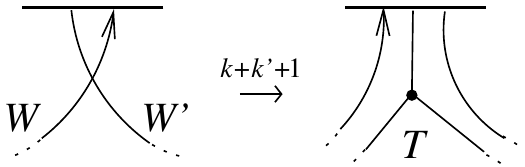}} \]
\noindent Here, $W$ and $W'$ are a $\w_{k}$-tree and a $\w_{k'}$-tree, respectively, for some $k,k'\ge 1$, and $T$ is a $\w_{k+k'}$-tree as shown. 
\end{lemma}

\begin{proof}
Consider the path of edges of $W'$ connecting the tail shown in the figure to the head, and denote by $n$ the number of edges in  this path: we have $1\le n\le k'$.  
The proof is by induction on $n$. 
More precisely, we prove by induction on $n$ the following stronger statement. 
 \begin{claim} \label{claim:th}
  Let $k,k'\ge 2$. Let $W$, $W'$ and $T$ be as above. The following equality holds. 
   \[ \textrm{\includegraphics[scale=0.9]{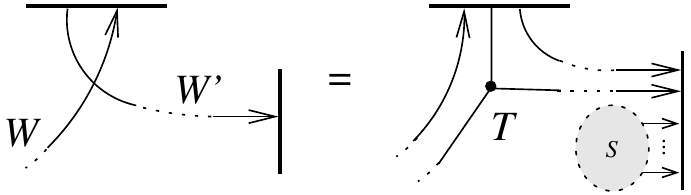}} \]
  \noindent where $S$ denotes a union of $\w$-trees of degree $>k+k'$.  
 \end{claim}
\noindent 
The case $n=1$ of the claim is a consequence of the Head--Tails Exchange Lemma \ref{lem:ht}, Claim \ref{claim:inverse} and the involutivity of twists. The proof of the inductive step is illustrated in Figure \ref{fig:thproof} below.  
\begin{figure}[!h]
   \includegraphics[scale=1]{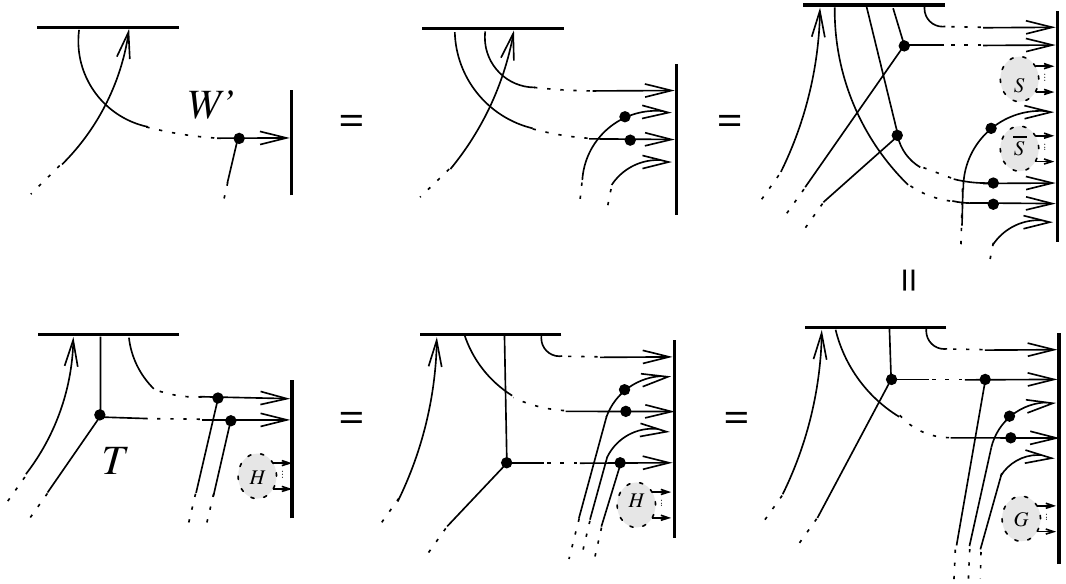}
  \caption{Here, $S$ (resp. $G$ and $H$) represent a union of w-trees of degree $>k+k'$ (resp. degree $>k+k'+1$) 
  }\label{fig:thproof}
\end{figure}
\noindent The first equality in Figure \ref{fig:thproof} is an application of (E) to the $\w_{k'}$-tree $W'$, while the second equality uses the induction hypothesis. The third (vertical) equality then follows from recursive applications of the Heads Exchange Lemma \ref{lem:he}, and uses also Convention \ref{rem:inverse2}. Further Heads Exchanges give the fourth equality, and the final one is given by (E). 
\end{proof}

We note the following consequence of Lemma \ref{lem:he} and Claim \ref{claim:th}. 
\begin{corollary}\label{lem:exchange}
Let $T$ and $T'$ be two w-trees, of degree $k$ and $k'$. 
We can exchange the relative position of two adjacent endpoints of $T$ and $T'$,  
at the expense of additional w-trees of degree $\ge k+k'$. 
\end{corollary}
\begin{proof}
There are three types of moves to be considered. 
First, exchanging two tails can be freely performed by the Tails Exchange move (3). 
Second, it follows from the Heads Exchange Lemma \ref{lem:he} 
that exchanging the heads of these two $\w$-trees can be performed at the cost of one $\w_{k+k'}$-tree.  
Third, by Claim \ref{claim:th}, exchanging a tail of one of these w-trees and the head of the other can be achieved up to addition of $\w$-trees of degree $\ge k+k'$. 
\end{proof}

Let us also note, for future use, the following consequence of these Exchange results. 
We denote by $\1_n$ the trivial $n$-component string link diagram.
\begin{lemma}\label{cor:separate}
Let $k,~l$ be integers such that $k\ge l\ge 1$. 
Let $W$ be a union of w-trees for $\1_n$ of degree $\ge l$. 
Then 
 $$ (\1_n)_{W} \stackrel{k+1}{\sim} (\1_n)_{T_1}\cdot \ldots \cdot (\1_n)_{T_{m}}\cdot (\1_n)_{W'}, $$
 where $T_1, \ldots, T_m$ are $\w_l$-trees and $W'$ is a union of w-trees of degree in $\{l+1,\cdots,k\}$. 
\end{lemma}
\begin{proof}
 This is shown by repeated applications of Corollary \ref{lem:exchange}. 
 More precisely, we use Exchange moves to rearrange the $\w_l$-trees $T_1, \ldots, T_{m}$ in $W$ so that they sit in disjoint disks $D_i$ ($i=1,\ldots,m$),  which intersects each component of $\1_n$ at a single trivial arc,  
 so that $(\1_n)_{\cup_i T_i} = (\1_n)_{T_1}\cdot \ldots \cdot (\1_n)_{T_{m}}$.  
 By Corollary \ref{lem:exchange}, this is achieved at the expense of w-trees of degree $\ge l+1$, which may intersect those disks. 
 But further Exchange moves allow to move all higher degree w-trees \emph{under} $\cup_i D_i$, according to the orientation of $\1_n$,  now at the cost of additional w-trees of degree $\ge l+2$, which possibly intersect $\cup_i D_i$. 
 We can repeat this procedure until the only higher degree w-trees intersecting $\cup_i D_i$ have degree $> k$, 
 which gives the desired equivalence. 
\end{proof}

Finally, we give a w-tree version of the IHX relation.
\begin{lemma}[IHX] \label{lem:ihx}
The following holds. 
\begin{figure}[!h]
   \includegraphics[scale=1]{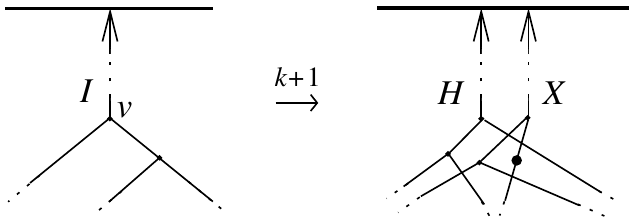}
  \caption{The IHX relation for w-trees}\label{fig:ihx}
\end{figure}

\noindent Here, $I$, $H$ and $X$ are three $\w_{k}$-tree for some $k\ge 3$.
\end{lemma}
\begin{proof}
We prove this lemma using the algebraic formalism of Section \ref{sec:algebra}, for simplicity. 
We prove the following stronger version.
 \begin{claim} \label{claim:ihx}
 For all $k\ge 2$, we have 
 $$ \underbrace{[\mydot ,[A,[B,C]]\mydots ]}_{\in\, \Gamma_{k} F}  = S\, [\mydot ,[[A,B],C]\mydots ][\mydot ,[[A,C],\ov{B}]\mydots ], $$
 \noindent for some $S\in \Gamma_{k+1} F$. 
 \end{claim}
The proof is by induction on the depth $d$ of $[A,[B,C]]$ in the iterated commutator $[\mydot ,[A,[B,C]]\mydots ]$, as defined in the proof of Claim \ref{claim:inverse}. 
Recall that, diagrammatically, the depth of $[A,[B,C]]$ is the number of edges connecting the vertex $v$ in Figure \ref{fig:ihx} to the head. 
The case $d=0$ is given by 
\begin{align*}
[A,[B,C]] & = A\, \ov{C}\, B\, C\, \ov{B}\, \ov{A}\, B\, \ov{C}\, \ov{B}\, C  & \\
              & = A\, \ov{C}\, B\, C\, \ov{A}\, [A,B]\, \ov{C}\, \ov{B}\, C & \\ 
              & = A\, \ov{C}\, B\, C\, \ov{A}\, \left[[A,B],C\right]\, \ov{C}\, [A,B]\, \ov{B}\, C & \\ 
              & = R'\, \left[[A,B],C\right]\, A\, \ov{C}\, B\, C\, \ov{A}\, \ov{C}\, [A,B]\, \ov{B}\, C & \textrm{(Heads Exchange Lem.~\ref{lem:he})}\\
              & = R'\, \left[[A,B],C\right]\, A\, \ov{C}\, B\, C\, \ov{A}\, \ov{C}\, A\, \ov{B}\, \ov{A}\, C & \\ 
              & = R'\, \left[[A,B],C\right]\, A\, \ov{C}\, B\, [C,A]\, \ov{B}\, \ov{A}\, C & \\ 
              & = R'\, \left[[A,B],C\right]\, A\, \ov{C}\, [C,A]\, \left[[\ov{A},\ov{C}],\ov{B}\right]\, \ov{A}\, C & \\ 
              & = R\, \left[[A,B],C\right]\, \left[[\ov{A},\ov{C}],\ov{B}\right] & \textrm{(Heads Exchange Lem.~\ref{lem:he})}\\
              & = R\, \left[[A,B],C\right]\, S'\, \left[[A,C],\ov{B}\right], & \textrm{(Twist Lem.~\ref{lem:twist})} \\
              & = S\, \left[[A,B],C\right]\, \left[[A,C],\ov{B}\right], & \textrm{(Heads Exchange Lem.~\ref{lem:he})}
\end{align*}
where $R,R',S'$ and $S$ are some elements of $\Gamma_{k+1} F$. 

For the inductive step, 
let $I'=[\mydot ,[A,[B,C]]\mydots ]$ be an element of $\Gamma_m F$, for some $m\ge 3$, such that $[A,[B,C]]$ has depth $d$, 
and set $H'=[\mydot ,[[A,B],C]\mydots ]$ and $X'=[\mydot ,[[A,C],\ov{B}]\mydots ]$. 
By the induction hypothesis, there exists an element $S\in \Gamma_{m+1}$ such that $I'= S\, H'\, X'$. 
Let $D\in \Gamma_l F$ such that $l+m=k$. 
Then we have 
\begin{align*}
 [D,I'] = D\, \ov{I'}\, \ov{D} I' 
         & = D\, \ov{X'}\, \ov{H'}\, \ov{S}\, \ov{D}\, S\, H'\, X' & \textrm{(induction hypothesis)} \\
         & = R'\, D\, \ov{X'}\, \ov{H'}\, \ov{D}\, H'\, X' & \textrm{(Heads Exchange Lem.~\ref{lem:he})} \\
         & = R'\, D\, \ov{X'}\, \ov{D}\, [D,H']\, X' &   \\
         & = R\,[D,H']\, D\, \ov{X'}\, \ov{D}\, X' & \textrm{(Heads Exchange Lem.~\ref{lem:he})} \\
         & = R\,[D,H']\, [D,X'],  &  
\end{align*}
where $R,R'$ are some elements in $\Gamma_{k+1} F$. 
\end{proof}

\section{Finite type invariants of welded knots and long knots}\label{sec:wk_knots}

We now use the $\w_k$-equivalence relation to characterize finite type invariants of welded (long) knots. 
Topological applications for surfaces in $4$-space are also given.  

\subsection{$\w_k$-equivalence for welded knots}

The fact, noted in Section \ref{sec:w1}, that any two welded knots are $\w_i$-equivalent for $i=1,2$, generalizes widely as follows.  
\begin{theorem}\label{thm:weldedknots}
\label{thm:wk}
 Any two welded knots are $\w_k$-equivalent, for any $k\ge 1$.
\end{theorem}
An immediate consequence is the following.
\begin{corollary}\label{cor:ftiwklong}
There is no non-trivial finite type invariant of welded knots. 
\end{corollary}
This was already noted for rational-valued finite type invariants by D.~Bar-Natan and S.~Dancso \cite{WKO1}. 
Also, we have the following topological consequence, which we show in Section \ref{sec:watanabe}.
\begin{corollary}\label{cor:topo1}
There is no non-trivial finite type invariant of ribbon torus-knots. 
\end{corollary}

Theorem \ref{thm:weldedknots} is a consequence of the following, stronger statement. 
\begin{lemma}\label{lem:separate} 
Let $k,~l$ be integers such that $k\ge l\ge 1$, and let 
$W$ be a union of $\w$-trees for $O$.
There is a union $W_l$ of $\w$-trees of degree $\ge l$ such that 
$O_W \stackrel{k+1}{\sim} O_{W_l}$.
\end{lemma}
\begin{proof}
The proof is by induction on $l$. 
The initial case, i.e., $l=1$ for any fixed integer $k\ge 1$, was given in Section \ref{sec:w1}.
Assume that $W$ is a union of $\w$-trees of degree $\ge l$. 
Using Lemma \ref{cor:separate}, we have that 
$O_{W}\stackrel{k+1}{\sim} O_{W'}$, where $W'$ a union of isolated $\w_l$-trees, and w-trees of degree in $\{l+1,\cdots,k\}$.
Here, a $\w_l$-tree for $O$ is called \emph{isolated} if it is contained in a disk $B$ which is disjoint from all other w-trees 
and intersects $O$ at a single arc. 

Consider such an isolated $\w_l$-tree $T$.
Suppose that, when traveling along $O$, the first endpoint of $T$ which is met in the disk $B$ is its head; then, up to applications of the Tails Exchange move (3) and Antisymmetry Lemma \ref{lem:as}, we have that $T$ contains a fork, 
so that it is equivalent to the empty w-tree by the Fork Lemma \ref{lem:fork}. 
Note that these moves can be done in the disk $B$.
If we first meet some tail when traveling along $O$ in $B$, we can slide this tail outside $B$ and use Corollary \ref{lem:exchange} 
to move it around $O$, up to addition of $\w$-trees of degree $\ge l+1$, until we can move it back in $B$. 
In this case, by Lemma \ref{cor:separate}, we may assume that the new w-trees of degree $\ge l+1$ do not intersect $B$ 
up to $\w_{k+1}$-equivalence. 
Using this and the preceding argument, we have that $T$ can be deleted up to $\w_{k+1}$-equivalence. 
This completes the proof.      
\end{proof}

\subsection{$\w_k$-equivalence for welded long knots}

We now turn to the case of long knots. 
In what follows, we use the notation $\1$ for the trivial long knot diagram (with no crossing).

As recalled in Section \ref{sec:w1}, it is known that any two welded long knots are $\w_i$-equivalent for $i=1,2$.  
The main result of this section is the following generalization. 
\begin{theorem}\label{thm:ftialexlong}
For each $k\ge 1$, welded long knots are classified up to $\w_k$-equivalence by the first $k-1$ normalized coefficients $\{ \alpha_i \}_{2\le i\le k}$ of the Alexander polynomial.
\end{theorem}
Since the normalized coefficients of the Alexander polynomial are of finite type, we obtain the following, which in particular gives the converse to Proposition \ref{prop:ftiwk} for welded long knots. 
\begin{corollary}\label{cor:ftiwk}
The following assertions are equivalent, for any integer $k\ge 1$: 
\begin{enumerate}
\item[$\circ$]  two welded long knots are $\w_k$-equivalent, 
\item[$\circ$]  two welded long knots share all finite type invariants of degree $<k$, 
\item[$\circ$]  two welded long knots have same invariants $\{ \alpha_i \}$ for $i< k$.
\end{enumerate}
\end{corollary}
Theorem \ref{thm:ftialexlong} also implies the following, which was first shown by K.~Habiro and A.~Shima \cite{HS}.  
\begin{corollary}\label{cor:topo2}
Finite type invariant of ribbon $2$-knots are determined by the (normalized) Alexander polynomial. 
\end{corollary}
\noindent Actually, we also recover a topological characterization of finite type invariants of ribbon $2$-knots due to T.~Watanabe, see Section \ref{sec:watanabe}. 

Moreover, by the multiplicative property of the normalized Alexander polynomial (Lemma \ref{lem:additive}), we have the following consequence. 
\begin{corollary}\label{cor:abelian}
Welded long knots up to $\w_k$-equivalence form a finitely generated free abelian group, for any $k\ge 1$. 
\end{corollary}
\noindent 

The proof of Theorem \ref{thm:ftialexlong} uses the next technical lemma, 
which refer to the welded long knots $L_k$ or $\overline{L_k}$ defined in Figure \ref{fig:K0}. Here we set $L_k^{-1}:=\overline{L_k}$.  
\begin{lemma} \label{lem:wklong}
Let $k,~l$ be integers such that $k\ge l\ge 1$,  
and let $W$ be a union of $\w$-trees of degree $\ge l$ for $\1$. 
Then 
$$ \1_W \stackrel{k+1}{\sim} (L_l)^x\cdot \1_{W'}, $$ 
for some $x\in \mathbb{Z}$, where $W'$ is a union of $\w$-trees of degree $\ge l+1$. 
\end{lemma}

Let us show how Lemma \ref{lem:wklong} allows to prove Theorem \ref{thm:ftialexlong}. 
\begin{proof}[Proof of Theorem \ref{thm:ftialexlong} assuming Lemma \ref{lem:wklong}]
We prove that, for any $k,l$ such that $k\ge l\ge 1$, a welded long knot $K$ satisfies 
\begin{equation}\label{eq:wk}
 K \stackrel{k+1}{\sim} \left(\prod_{i=2}^{l-1} L_i^{x_i(K)}\right)\cdot \1_{W_l}, 
\end{equation} 
where 
$(\1,W_l)\lr{l}(\1,\emptyset)$, and where 
$$x_i(K)=\left\{\begin{array}{ll}
                                         \alpha_i(K) & \textrm{ if $i=2$, }\\
                                         \alpha_i(K) - \alpha_i\left(\prod_{j=2}^{i-1} L_j^{x_j(K)}\right) & \textrm{ if $i>2$. }
                                        \end{array}
   \right.$$ 

\noindent We proceed by induction on $l$. Assume Equation (\ref{eq:wk}) for some $l\ge 1$
and any fixed $k\ge l$. 
By applying Lemma \ref{lem:wklong} to the welded long knot $\1_{W_l}$, we have 
$\1_{W_l}\stackrel{k+1}{\sim} (L_l)^{x}\cdot \1_{W_{l+1}}$, 
where $(\1,W_{l+1})\stackrel{\scriptsize{\raisebox{-.3ex}[0pt][-.3ex]{$l+1$}}}{\rightarrow}(\1,\emptyset)$. 
Using the additivity (Corollary \ref{cor:additive}) 
and finite type (Lemma \ref{lem:Alexfti} and Proposition \ref{prop:ftiwk}) properties of the normalized coefficients of the Alexander polynomial,  
we obtain that $x=x_l(K)$, thus completing the proof.
\end{proof}

\begin{proof}[Proof of Lemma \ref{lem:wklong}]
By Lemma \ref{cor:separate}, 
we may assume that 
 $$ 1_{W} \stackrel{k+1}{\sim} \1_{T_1}\cdot \ldots \cdot \1_{T_{m}}\cdot \1_{W'}, $$
where each $T_i$ is a single $\w_l$-tree and $W'$ is a union of  w-trees of degree in $\{l+1,\cdots,k\}$.

Consider such a $\w_l$-tree $T_i$. 
Let us call `external' any vertex of $T_i$ that is connected to two tails. In general, $T_i$ might contain several external vertices, 
but by the IHX Lemma \ref{lem:ihx} and Lemma \ref{cor:separate}, 
we can freely assume that $T_i$ has only one external vertex, up to $\w_{k+1}$-equivalence. 

By the Fork Lemma \ref{lem:fork} and the Tails Exchange move (3), if the two tails connected to this vertex are not separated by the head, then $T_i$ is equivalent to the empty w-tree. Otherwise, using the Tails Exchange move, 
we can assume that these two tails are at the leftmost and rightmost positions among all endpoints of $T_i$ along $\1$, 
as for example for the $\w_l$-tree shown in Figure \ref{fig:Twk}. 
The result then follows from the observation shown in this figure. 
\begin{figure}[!h]
    \includegraphics[scale=0.9]{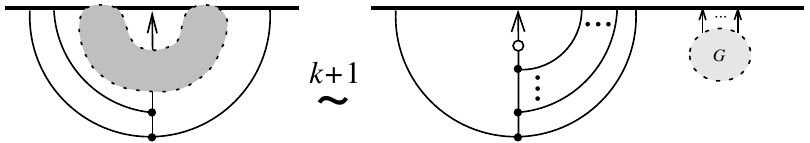}
    \caption{The shaded part contains all unrepresented edges of the $\w_l$-tree, and $G$ is a union of w-trees of degree in $\{l+1,\cdots,k\}$
    }\label{fig:Twk}
\end{figure}
\noindent Indeed, combining these equalities with the involutivity of twists and the Twist Lemma \ref{lem:twist}, 
we have that $T_i$ can be deformed, up to $\w_{k+1}$-equivalence, into one of the two $\w_l$-trees of Figure \ref{fig:K0}, at the cost of adding a union of w-trees of degree in $\{l+1,\cdots,k\}$. 

Let us prove the equivalence of Figure \ref{fig:Twk}.  
To this end, consider the union $A\cup F$ of a w-arrow $A$ and a $\w_{l-1}$-tree $F$ as shown on the left-hand side of Figure \ref{fig:Twk2}. 
On one hand, by the Fork Lemma \ref{lem:fork}, followed by the the Isolated move (4), we have that $\1_{A\cup F}=\1$. 
\begin{figure}[!h]
   \includegraphics[scale=1]{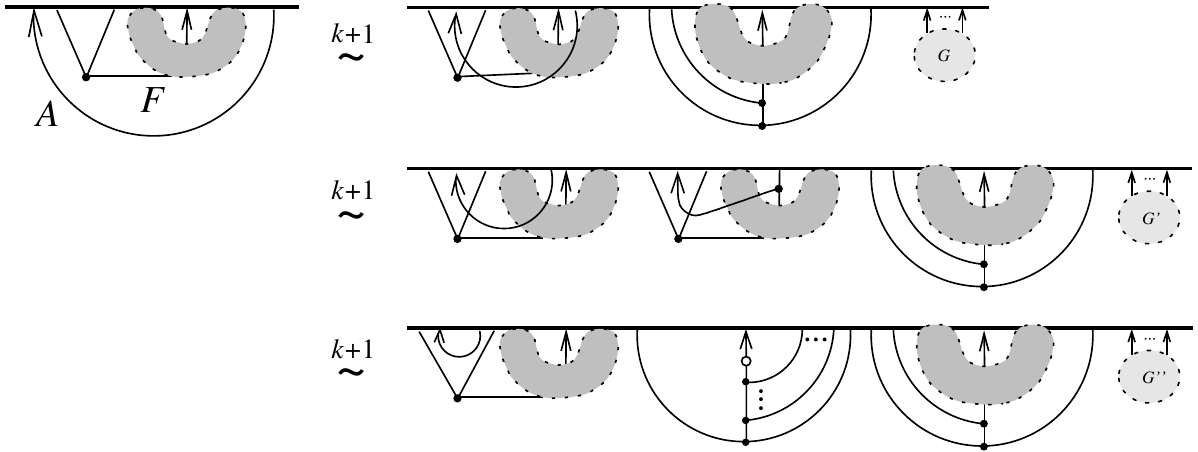}
  \caption{ \\ Here, $G, G', G''$ are unions of w-trees of degree in $\{l+1,\cdots,k\}$
  }\label{fig:Twk2}
\end{figure}
On the other hand, we can use the Head--Tail Exchange Lemma \ref{lem:th} to move the head of $A$ across the adjacent tail of $F$, 
and apply the Tails Exchange move (3) to move the tail of $A$ towards the head of $F$, thus producing, by Lemma \ref{cor:separate}, the first equivalence
of Figure \ref{fig:Twk2}.  
We can then apply the Head--Tail Exchange Lemma to move the head of $A$ across the head of $F$, which by Lemma \ref{cor:separate} yields the second equivalence. 
Further applications of Lemma \ref{cor:separate}, together with the Antisymmetry and Twist Lemmas \ref{lem:as} and \ref{lem:twist},
give the third equivalence. Finally, the first term in the right-hand side of this equivalence is trivial by the Isolated move (4) and the Fork Lemma. 
The equivalence 
of Figure \ref{fig:Twk} is then easily deduced, using the Inverse Lemmas \ref{lem:inverse} and \ref{cor:separate}.  
\end{proof}

\section{Homotopy arrow calculus}\label{sec:homotopy}

The previous section shows how the study of welded knotted objects of one components is well-understood when working up to $\w_k$-equivalence. 
The case of several components (welded links and string links), though maybe not out of reach, is significantly more involved. 

One intermediate step towards a complete understanding of knotted objects of several components is to study these objects `modulo knot theory'. 
In the context of classical (string) links, this leads to the notion of \emph{link-homotopy}, were each individual component is allowed to cross itself; 
this notion was first introduced by Milnor \cite{Milnor}, and culminated with the work of Habegger and Lin \cite{HL} who used Milnor invariants to classify string link up to link-homotopy.  
In the welded context, the analogue of this relation is generated by the \emph{self-virtualization move}, where a crossing involving two strands of a same component 
can be replaced by a virtual one. 
In what follows, we simply call \emph{homotopy} this equivalence relation on welded knotted objects, which we denote by $\stackrel{\textrm{h}}{\sim}$. 
This is indeed a generalization of link-homotopy, 
since a crossing change between two strands of a same component can be generated by two self-(de)virtualizations. 

We have the following natural generalization of \cite[Thm.~8]{Milnor2}. 
\begin{lemma}\label{lem:milnorhomotopy}
If $I$ is a sequence of non repeated indices, then $\mu^w_I$ is invariant under homotopy. 
\end{lemma}
\begin{proof}
The proof is essentially the same as in the classical case. 
Set $I=i_1\cdots i_m$, such that $i_j\neq i_k$ if $j\neq k$. 
It suffices to show that $\mu^w_I$ remains unchanged when a self-(de)virtualization move is performed on the $i$th component, which is done by distinguishing two cases. 
If $i=i_m$, then the effect of this move on the combinatorial $i_m$th longitude is multiplication by an element of the normal subgroup $N_{i}$ generated by $m_{i}$; each (non trivial) term in the Magnus expansion of such an element necessarily contains 
$X_{i_m}$ at least once, and thus  $\mu^w_I$ remains unchanged. 
If $i\neq i_m$, then this move can only affect the combinatorial $i_m$th longitude by multiplication by an element of $[N_i,N_i]$: any non trivial term in the Magnus expansion of such an element necessarily contains $X_{i}$ at least twice. 
\end{proof}

\subsection{w-tree moves up to homotopy}\label{sec:hw}

Clearly, the w-arrow incarnation of a self-virtualization move is the deletion of a w-arrow whose tail and head are attached to a same component. In what follows, we will call such a w-arrow a \emph{self-arrow}. 
More generally, a \emph{repeated} w-tree is a w-tree having two endpoints attached to a same component of a diagram. 
\begin{lemma}\label{lem:h}
 Surgery along a repeated w-tree does not change the homotopy class of a diagram. 
\end{lemma}
\begin{proof}
Let $T$ be a w-tree having two endpoints attached to a same component. 
We must distinguish between two cases, depending on whether these two endpoints contain the head of $T$ or not.\\
\textbf{Case 1:} The head and some tail $t$ of $T$ are attached to a same component.
Then we can simply expand $T$: the result contains a bunch of self-arrows, joining (a neighborhood of) $t$ to (a neighborhood of) the head of $T$. By the Brunnian-type property of w-trees (Remark \ref{rem:expansion}), deleting all these self-arrows yields a union of w-arrows which is equivalent to the empty one.\\
\textbf{Case 2:} Two tails $t_1$ and $t_2$ of $T$ are attached to a same component. 
Consider the path of edges connecting these two tails, 
and denote by $n$ the number of edges connecting this path to the head: 
we proceed by induction on this number $n$. 
The case $n=1$ is illustrated in Figure \ref{fig:hproof}.  
As the first equality shows, one application of $(E)$ yields four w-trees $T_1, \ov{T_1},T_2, \ov{T_2}$. 
\begin{figure}[!h]
   \includegraphics[scale=1]{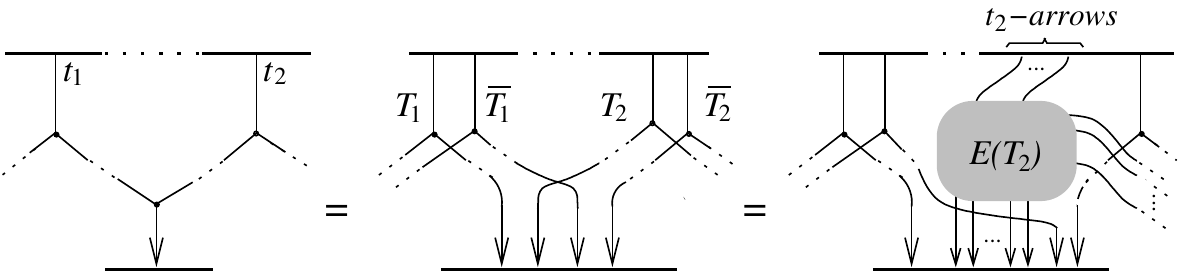}
  \caption{Proof of Lemma \ref{lem:h}}\label{fig:hproof}
\end{figure}
For the second equality, expand the w-tree $T_2$, and denote by $E(T_2)$ the result of this expansion. 
Let us call `$t_2$-arrows' the w-arrows in $E(T_2)$ whose tail lie in a neighborhood of $t_2$. We can successively slide all other w-arrows in   $E(T_2)$ along the $t_2$-arrows, and next slide the two w-trees $T_1$ and $\ov{T_1}$, using Remark \ref{rem:genslide}: 
the result is a pair of repeated w-trees as in Case 1 above, which we can delete up to homotopy. 
Reversing the slide and expansion process in $E(T_2)$, we then recover $T_2\cup \ov{T_2}$, which can be deleted by the Inverse Lemma \ref{lem:inverse}. 
The inductive step is clear, using (E) and the Inverse Lemma \ref{lem:inverse}. 
\end{proof}

\begin{remark}
Thanks to the previous result, the lemmas given in Section \ref{sec:tech} for w-tree presentations still hold when working up to homotopy. 
More precisely, Lemmas \ref{lem:twist}, \ref{lem:th} and \ref{lem:ihx} remain valid when using, in the statement, the notation $\stackrel{\textrm{h}}{\sim}$. 
This is a consequence of the proofs of Claims \ref{claim:inverse}, \ref{claim:th} and \ref{claim:ihx}, which show that the equality in these lemmas is achieved by surgery along repeated w-trees. 
In what follows, we will implicitly make use of this fact, and freely refer to the lemmas of the previous sections when using their homotopy versions. 
\end{remark}

\subsection{Homotopy classification of welded string links}\label{sec:hsl}

Let $n\ge 2$. 
For each integer $i\in \{1,\cdots,n\}$, 
denote by $\s_l(i)$ the set of all sequences $i_1\cdots i_l$ of $l$ distinct integers from $\{1,\cdots ,n\}\setminus \{i\}$ such that $i_j<i_l$ for all $j=1,\ldots,l-1$. Note that the lexicographic order endows the set $\s_l(i)$ with a total order. 

For any sequence $I=i_1\cdots i_{k-1}\in \s_{k-1}(i)$,  
consider the $\w_{k-1}$-trees $T_{Ii}$ and $\overline{T_{Ii}}$ for the trivial diagram $\mathbf{1}_n$ introduced in Lemma \ref{lem:Milnor}. 
Set 
$$ W_{Ii} := (\1_n)_{T_{Ii}}\quad \textrm{and} \quad W_{Ii}^{-1} := (\1_n)_{\overline{T_{Ii}}}. $$
We prove the following (compare with Theorem 4.3 of \cite{yasuhara}). 
This gives a complete list of representatives for welded string links up to homotopy. 
\begin{theorem}\label{thm:hrep}
Let $L$ be an $n$-component welded string link. 
Then $L$ is homotopic to $l_1\cdots l_{n-1}$, where for each $k$, 
$$ l_k = \prod_{i=1}^n \prod_{I\in \s_{k}(i)} \left(  W_{Ii} \right)^{x_I}, \textrm{ where }x_I=\left\{\begin{array}{ll}
\mu^w_{ji}(L) & \textrm{if $k=1$ and $I=j$,}\\
\mu^w_{Ii}(L) - \mu^w_{Ii}(l_1\cdots l_{k-1}) & \textrm{if $k>1$.} 
\end{array}
\right.$$
\end{theorem}

As a consequence, we recover the following classification results.
\begin{corollary}\label{thm:wsl}
Welded string links are classified up to homotopy by welded Milnor invariants indexed by non-repeated sequences. 
\end{corollary}
Corollary \ref{thm:wsl} was first shown by Audoux, Bellingeri, Wagner and the first author in \cite{ABMW}: 
their proof consists in defining a global map from welded string links up to homotopy to conjugating automorphisms of the reduced free group,  then to use Gauss diagram to build an inverse map. 
\begin{remark}
Corollary \ref{thm:wsl} is a generalization of the link-homotopy classification of string links by Habegger and Lin \cite{HL}: it is indeed shown in \cite{ABMW} that string links up to link-homotopy  embed in welded string links up to homotopy. 
However, Theorem \ref{thm:hrep} does not allow to recover the result of \cite{HL}.
By Remark \ref{rem:classico}, it only implies that two classical string link diagrams are related by a sequence of isotopies and self-(de)virtualizations if and only if they have same Milnor invariants. 
\end{remark}

\begin{proof}[Proof of Theorem \ref{thm:hrep}]
Let $L$ be an $n$-component welded string link. 
Pick an Arrow presentation for $L$. 
By Lemma \ref{cor:separate}, we can freely rearrange the w-arrows up to $\w_n$-equivalence, so that 
$$ L \stackrel{n}{\sim} \prod_{j\neq i} \left(  W_{ji} \right)^{x_{ji}}\cdot (\1_n)_{R_1}\cdot  (\1_n)_{S_{\ge 2}}, $$ 
for some integers $x_{ji}$, where $R_1$ is a union of self-arrows, and $S_{\ge 2}$ is a union of w-trees of degree in $\{2,\cdots,n-1\}$.  
Up to homotopy, we can freely delete all self-arrows, and using the properties of Milnor invariants 
(Lemmas \ref{lem:Madditive} and \ref{lem:Milnor}, Remark \ref{rem:Milnorwk}, and Lemma \ref{lem:milnorhomotopy}), 
we have that $x_{ji}=\mu^w_{ji}(L)$ for all $j\neq i$.  
Hence we have 
$$ L\stackrel{\textrm{h}}{\sim}  l_1\cdot  (\1_n)_{S_{\ge 2}}. $$ 
Next, we can separate, by a similar procedure, all $\w_2$-trees in $S_{\ge 2}$. 
Repeated $\w_2$-trees can be deleted thanks to  Lemma \ref{lem:h}. 
Next, we need the following general fact,\footnote{This is merely a w-tree version of the well-known fact that any Jacobi tree diagram can be written, up to AS and IHX, as a linear sum of Ôlinear' tree diagrams, see e.g. \cite[Fig.~3]{HaMe1}. }  which is easily checked using the Antisymmetry, IHX and Twist Lemmas \ref{lem:as}, \ref{lem:ihx} and \ref{lem:twist}. 
\begin{claim}\label{claim:sep}
Let $T$ be a non repeated $\w_k$-tree for $\1_n$ ($k\ge 2$), such that the head of $T$ is attached to the $i$th strand of $\1_n$. 
Then 
 $$ (\1_n)_T \stackrel{h}{\sim}  \prod_{i=1}^N (\1_n)_{T_i}, $$
for some $N\ge 1$, where each $T_i$ is a copy of either $T_{Ii}$ or $\overline{T}_{Ii}$ for some $I\in \mathcal{S}_{k-2}(i)$. 
\end{claim}
Hence, we obtain
$$ L\stackrel{h}{\sim} l_1\cdot  \prod_{i=1}^n \prod_{I\in \s_{2}(i)} \left(  W_{Ii} \right)^{x_I} \cdot (\1_n)_{S_{\ge 3}}, $$ 
for some integers $x_I$, where $S_{\ge 3}$ is a union of w-trees of degree in $\{ 3\cdots,n-1\}$. 
By using the properties of Milnor invariants, we have 
\begin{eqnarray*}
\mu^w_{Ii}(L)  & = & \mu^w_{Ii}\left( l_1\cdot  \prod_{i=1}^n \prod_{I\in \s_{2}(i)} \left(  W_{Ii} \right)^{x_I} \right) \\
                       & = & \mu^w_{Ii}(l_1) + \sum_{i=1}^n \sum_{I\in \s_{2}(i)} {x_I} \mu^w_{Ii}\left(  W_{Ii} \right) \\
                       & = & \mu^w_{Ii}(l_1) + x_{I},  
\end{eqnarray*}
thus showing that 
$$ L\stackrel{\textrm{h}}{\sim} l_1\cdot  l_2\cdot  (\1_n)_{S_{\ge 3}}. $$ 
Iterating this procedure, using Claim \ref{claim:sep} and the same properties of Milnor invariants, we eventually obtain that 
$L\stackrel{n}{\sim} l_1\cdots l_{n-1}$. 
The result follows by Lemma \ref{lem:h}, since w-trees of degree $\ge n$ for $\1_n$ are necessarily repeated.
\end{proof}

\begin{remark}
It was shown in \cite{ABMW} that Corollary \ref{thm:wsl}, together with the Tube map, gives homotopy classifications of ribbon tubes and ribbon torus-links (see Section \ref{sec:ribbon}). 
Actually, we can deduce easily a homotopy classification of ribbon string links in codimension $2$, in any dimension, see \cite{AMW}. 
\end{remark}

\section{Concluding remarks and questions}\label{sec:thisistheend}

\subsection{Welded arcs}\label{sec:arcs}

There is yet another class of welded knotted object that we should mention here. 
A \emph{welded arc} is an immersed oriented arc in the plane, up to generalized Reidemeister moves, OC moves, and the additional move of Figure \ref{fig:arc} (left-hand side). There, we represent the arc endpoints by large dots. 
We emphasize that these large dots are `free' in the sense that they can be freely isotoped in the plane. 
It can be checked that welded arcs have a well-defined composition rule, 
given by gluing two arc endpoints, respecting the orientations.  
This is actually a very natural notion from the $4$-dimensional point of view, see Section \ref{sec:watanabe} below. 
\begin{figure}[!h]
   \includegraphics[scale=0.9]{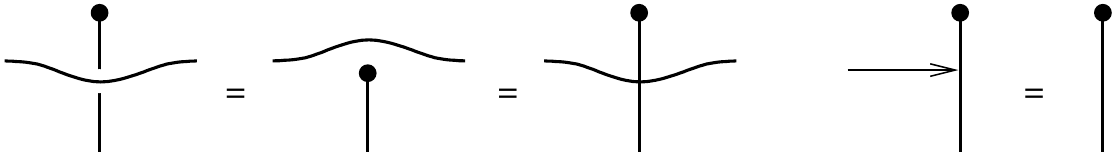}
  \caption{Additional moves for welded arcs, and the corresponding extra w-tree move}\label{fig:arc}
\end{figure}

Figure \ref{fig:arc} also gives the additional move for welded arcs in terms of w-trees: we can freely delete a w-tree whose head is adjacent to an arc endpoint. 
This is reminiscent of the case of welded long knots. Indeed, if a welded long knot is obtained from the trivial diagram $\1$ by surgery along a w-tree $T$ whose head is adjacent to an endpoint of $\1$, then by the Fork Lemma \ref{lem:fork}, 
we have $\1_T=\1$. This was observed in the proof of Lemma  \ref{thm:ftialexlong}. 
A consequence is that the proof of this lemma can be applied verbatim to welded arcs (in particular, the key fact of Figure \ref{fig:Twk} applies).
This shows that welded arcs up to $\w_k$-equivalence form an abelian group, which is isomorphic to that of welded long knots up to $\w_k$-equivalence, for any $k\ge 1$.  
Finite type invariants of welded arcs are thus classified similarly. 

To be more precise, there is a natural \emph{capping} map $C$ from welded long knots to welded arcs, which replaces the (fixed) endpoints by (free) large dots. This map $C$ is clearly surjective and the above observation says that it induces a bijective map when working up to $w_k$-equivalence. 
It seems however unknown whether the map $C$ itself is injective. 

\subsection{Finite type invariants of ribbon $2$-knots and torus-knots}\label{sec:watanabe}

As outlined above, the notion of welded arcs is relevant for the study of ribbon $2$-knots in $4$-space. 
Indeed, applying the Tube map to a welded arc, capping off by disks at the endpoints, yields a ribbon $2$-knot, and any ribbon $2$-knot arises in this way \cite{Satoh}. 
Combining this with the surjective map $C$ from Section \ref{sec:arcs} above, we obtain: 
\begin{fact}\label{fact1}
Any ribbon $2$-knot can be presented, via the Tube map, by a welded long knot. 
\end{fact}

Recall that K.~Habiro introduced in \cite{Habiro} the notion of $C_k$-equivalence, and more generally the calculus of claspers, 
and proved that two knots share all finite type invariants of degree $<k$ if and only if they are $C_k$-equivalent. 
As a $4$-dimensional analogue of this result, T.~Watanabe introduced in \cite{watanabe} the notion of $RC_k$-equivalence, and a topological calculus for ribbon $2$-knots. He proved the following.
\begin{theorem}\label{thm:watanabe}
Two ribbon $2$-knots share all finite type invariants of degree $<k$ if and only if they are $RC_k$-equivalent. 
\end{theorem}

We will not recall the definition of the $RC_k$-equivalence here, but only note the following.
\begin{fact}\label{fact2}
If two welded long knots are $w_k$-equivalent, then their images by the Tube map are $RC_k$-equivalent. 
\end{fact}
\noindent This follows from the definitions for $k=1$ (see Figure 3 of \cite{watanabe}), and can be verified using (E) and Watanabe's moves \cite[Fig.~6]{watanabe} for higher degrees. 

Corollary \ref{cor:ftiwk} gives a welded version of Theorem \ref{thm:watanabe}, and can actually be used to reprove it. 
\begin{proof}[Proof of Theorem \ref{thm:watanabe}]
Let $R$ and $R'$ be two ribbon $2$-knots and, using Fact \ref{fact1}, let $K$ and $K'$ be two welded long knots representing $R$ and $R'$, respectively. 
If $R$ and $R'$ share all finite type invariants of degree $<k$, then they have same normalized coefficients of the Alexander polynomial $\alpha_i$ for $1<i<k$, by \cite{HKS}. As seen in Remark \ref{rem:faithful}, this means that $K$ and $K'$ have same $\alpha_i$ for $1<i<k$, hence are $\w_k$-equivalent by Corollary \ref{cor:ftiwk}. 
By Fact \ref{fact2}, this shows that $R$ and $R'$ are $RC_k$-equivalent, as desired. (The converse implication is easy, see \cite[Lem.~5.7]{watanabe}). 
\end{proof}

Using very similar arguments, we now provide quick proofs for the topological consequences of Corollaries \ref{cor:ftiwk} and \ref{cor:ftiwklong}. 

\begin{proof}[Proof of Corollary \ref{cor:topo2}]
If two ribbon $2$-knots have same invariants $\alpha_i$ for $1<i<k$, 
then the above argument using Corollary \ref{cor:ftiwk} shows that they are $RC_k$-equivalent. This implies that they 
cannot be distinguished by any finite type invariant (\cite[Lem.~5.7]{watanabe}).   
\end{proof}

\begin{proof}[Proof of Corollary \ref{cor:topo1}]
Let $T$ be a ribbon torus-knot. In order to show that $T$ and the trivial torus-knot share all finite type invariants, it suffices to show that they are $RC_k$-equivalent for any integer $k$. 
But this is now clear from Fact \ref{fact2}, since any welded knot $K$ such that $\Tube(K)=T$ is $\w_k$-equivalent to the trivial diagram, by Theorem \ref{thm:weldedknots}. 
\end{proof}

\subsection{Welded string links and universal invariant}\label{sec:wslfti}

We expect that Arrow calculus can be successfully used to study welded string links, beyond the homotopy case treated in Section \ref{sec:homotopy}. 
In view of Corollary \ref{cor:ftiwklong}, and of Habiro's work in the classical case \cite{Habiro}, 
it is natural to ask whether finite type invariants of degree $<k$ classify welded string links up to $\w_k$-equivalence. 
A study of the low degree cases, using the techniques of \cite{MY3}, seem to support this fact. 

A closely related problem is to understand the space of finite type invariants of weldeds string links. 
One can expect that there are essentially no further invariants than those studied in this paper, i.e. that the normalized Alexander polynomial and welded Milnor invariants together provide a universal finite type invariant of welded string links. 
One way to attack this problem, at least in the case of rational-valued invariants, is to relate those invariants to the universal invariant $Z^w$ of D.~Bar-Natan and Z.~Dancso \cite{WKO1}. 
It is actually already shown in \cite{WKO1} that $Z^w$ is equivalent to the normalized Alexander polynomial for welded long knots, 
and it is very natural to conjecture that the `tree-part' of $Z^w$ is equivalent to welded Milnor invariants, in view of the classical case \cite{HM}. 
Observe that, from this perspective, w-trees appear as a natural tool, as they provide a `realization' of the space of oriented diagrams where $Z^w$ takes its values (see also \cite{polyak_arrow}), 
just like Habiro's claspers realize Jacobi diagrams for classical knotted objects. 
In this sense, Arrow calculus provides the Goussarov-Habiro theory for welded knotted objects. 

\subsection{$\w_n$-equivalence versus $C_n$-equivalence}\label{sec:cneq}

Recall that, for $n\ge 1$, a $C_n$-move is a local move on knotted objects involving $n+1$ strands, as shown in Figure \ref{fig:cn}.
\begin{figure}[!h]
 \includegraphics[scale=0.9]{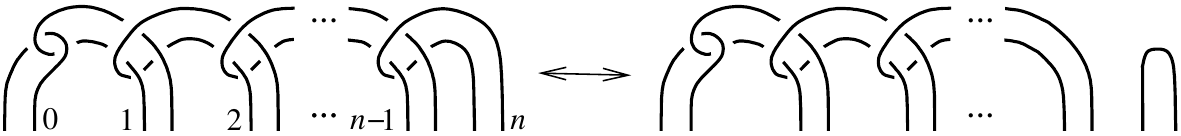}
 \caption{A $C_n$-move.}\label{fig:cn}
\end{figure} (A $C_1$-move is by convention a crossing change.)
The $C_n$-equivalence is the equivalence relation generated by $C_n$-moves and isotopies. 
\begin{proposition}\label{prop:wc}
For all $n\ge 1$, $C_n$-equivalence implies $\w_n$-equivalence. 
\end{proposition}
\begin{proof}
It suffices to show that a $C_n$-move can be realized by surgery along w-trees of degree $\ge n$, which is done by induction. 
It is convenient to use the following notion; given a w-tree $T$ for a diagram with components labeled from $0$ to $n$, the  
\emph{index} of $T$ is the set of all indices $i$ such that $T$ intersects the $i$th component of $D$ at some endpoint.  
We prove the following. 
\begin{claim}\label{claim:cnwn}
For all $n\ge 1$, the diagram shown on the left-hand side of Figure \ref{fig:cn} is obtained from the $(n+1)$-strand trivial diagram 
by surgery along a union $F_n$ of w-trees, such that each component of $F_n$ has index $\{0,1,\cdots,i\}$ for some $i$. 
\end{claim}
Before showing Claim \ref{claim:cnwn}, let us observe that it implies Proposition \ref{prop:wc}. 
Note that, if we delete those $w$-trees in $F_n$ having index $\{0,1,...,n\}$, we obtain a $\w$-tree presentation of the right-hand side of Figure \ref{fig:cn}. Such w-trees have degree $\geq n$, 
and by the Inverse Lemma \ref{lem:inverse}, deleting them can be realized by surgery along w-trees of degree $\geq n$.  
Therefore we have shown Proposition \ref{prop:wc}. 

Let us now turn to the proof of Claim \ref{claim:cnwn}. 
The case $n=1$ is clear, since it was already noted that a crossing change can be achieved by a sequence of (de)virtualization moves or, equivalently, by surgery along w-arrows (see Section \ref{sec:w1}). 
Now, using the induction hypothesis, consider the following w-tree presentation for the $(n+1)$-strand diagram on the left-hand side of Figure \ref{fig:cn}: 
 $$ \textrm{ \includegraphics[scale=1]{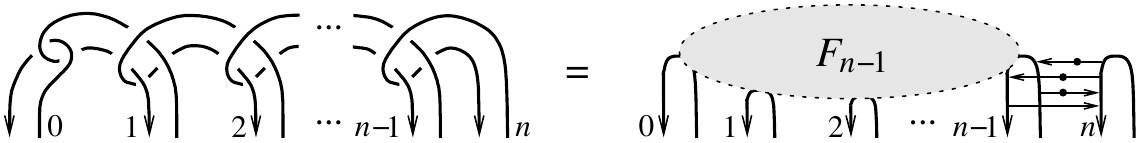}} $$
 
\noindent (Here, we have made a choice of orientation of the strands, but it is not hard to check that other choices can be handled similarly.) 
By moving their endpoints accross $F_{n-1}$, the four depicted w-arrows with index $\{n-1,n\}$ can be cancelled pairwise. 
By Corollary \ref{lem:exchange}, moving w-arrow ends accross $F_{n-1}$ can be made at the expense of additional w-trees with index $\{0,1,\cdots,n\}$. This completes the induction.
\end{proof}

\subsection{Arrow presentations allowing classical crossings}

In the definition of an Arrow presentation (Def. \ref{def:arrowpres}), we have restricted ourselves to diagrams with only virtual crossings. 
Actually, we could relax this condition, and consider more general Arrow presentations with both classical and virtual crossings. 
The inconvenience of this more general setting is that some of the moves involving w-arrows and crossings are not valid in general. 
For example, although passing a diagram strand \emph{above} a w-arrow tail is a valid move (as one can easily check using the OC move), passing \emph{under} a w-arrow tail is not permitted, as it would violate the forbidden UC move. 
Note that passing above or under a w-arrow head is allowed. 
Since one of the main interests of Arrow calculus resides, in our opinion, in its simplicity of use, we do not further develop this more general (and delicate) version in this paper. 

\subsection{Arrow calculus for virtual knotted objects}\label{sec:virtual}

Although we restricted ourselves to the study of welded (and classical) knotted objects, the main definitions of this paper can be applied verbatim for virtual ones. Specifically, we can use the notions of Arrow and w-tree presentations for virtual knotted objects. 
The main difference is that the corresponding calculus is more constrained, and significantly less simple in practice. 
Among the six Arrow moves of  Section \ref{sec:arrow_moves}, moves (1), (2), (4) and (5) are still valid, but the Tails Exchange move (3) is forbidden (we indeed saw that it is essentially equivalent to the OC move); the Slide move (6) is not valid in the given form, as the proof also uses the OC move, but a version can be given for virtual diagrams, which is closer in spirit to Gauss diagram versions of the Reidemeister III move. 
The calculus for w-trees is thus also significantly alterred. 

The $\w_k$-equivalence relation also makes sense for virtual objects. It is noteworthy that Proposition \ref{prop:ftiwk} still holds in this context: two virtual knotted objects that are $w_k$-equivalent ($k\ge 1$) cannot be distinguished by finite type invariants of degree $<k$. 
Indeed, as seen in Section \ref{sec:noteworthy}, the proof is mainly formal, and only uses the definition of w-trees, and more precisely their Brunnian property (Remark \ref{rem:expansion}). It would be interesting to study the converse implication for virtual (long) knots. 

From the universal invariant point of view, however, a full virtual extension of Arrow calculus should provide a diagrammatic realization of Polyak's algebra \cite{polyak_arrow}, which implies a significant enlargement; for example, vertices with one ingoing and two outgoing edges, and the moves involving such w-trees, should be investigated. 

\bibliographystyle{abbrv}
\bibliography{arrow}

\end{document}